\documentclass[12pt, a4paper]{amsart}
\usepackage[margin = 2.9cm]{geometry}
\usepackage[title]{appendix}
\usepackage{mathtools}
\usepackage{setspace}
\usepackage{textcomp}
\usepackage{enumerate}
\usepackage{graphicx}
\usepackage{subcaption}

\usepackage[dvipsnames]{xcolor}

\usepackage{amsmath}

\usepackage{graphics}
\usepackage[colorlinks=true]{hyperref}
\hypersetup{urlcolor=blue, citecolor=red}
\usepackage[dvipsnames]{xcolor}

\usepackage{lineno}

\newtheorem{thm}{Theorem}[section]
\newtheorem{corollary}{Corollary}

\newtheorem{lm}{Lemma}
\newtheorem{pp}{Proposition}
\newtheorem{claim}{Claim}

\theoremstyle{definition}
\newtheorem{df}{Definition}

\newtheorem{assumption}{Assumption}

\newcommand*\diff{\mathop{}\!\mathrm{d}}

\newcommand{\ep}{\epsilon}
\newcommand{\lge}{\langle}
\newcommand{\rge}{\rangle}
\newcommand{\zm}{{Z^m}}
\newcommand{\zmm}{{Z^{m-1}}}
\newcommand{\zi}{{\zeta^{(i)}}}
\newcommand{\zj}{{\zeta^{(j)}}}
\newcommand{\zk}{{\zeta^{(k)}}}
\newcommand{\zl}{{\zeta^{(l)}}}

\newcommand{\nxxi}{{\mathcal{N}(\vec{x}, \vec{\xi})}}
\newcommand{\ntxxi}{{{\mathcal{R}}(\vec{x}, \vec{\xi})}}

\newcommand{\unionGamma}{\Gamma({\vec{x}, \vec{\xi}})}

\newcommand{\capgamma}{{\cap_{j=1}^{4}\gamma_{x_j, \xi_j}(\mathbb{R}_+)}}

\newcommand{\tw}{{\widetilde{w}}}

\newcommand{\wfset}{{\text{\( \WF \)}}}
\newcommand{\tQ}{Q}

\newcommand{\bx}{x_|}
\newcommand{\bxi}{\xi_|}

\newcommand{\bTM}{{}^{b}\dot{T}^* M}

\newcommand{\intM}{M^{{o}}}
\newcommand{\Char}{\mathrm{Char}}

\newcommand{\gsetint}{\mathcal{G}^{\mathrm{int}}}
\newcommand{\compchar}{\dot{\Sigma}_g}

\newcommand{\Ical}{\mathcal{I}}
\newcommand{\Ir}{\mathring{\mathcal{I}}}

\newcommand{\sigmp}{{ \sigma_p}}

\newcommand{\pM}{\partial M}

\newcommand{\tM}{{M_{\mathrm{e}}}}

\newcommand{\Fk}{F^{(k)}}
\newcommand{\pk}{p^{(k)}}
\newcommand{\betak}{\beta^{(k)}}

\newcommand{\TMpm}{T_{\partial M, \pm} M}

\newcommand{\LcMpm}{L^{*}_{\partial M, \pm}M}
\newcommand{\LcMpo}{L^{*}_{\partial M, +}M}

\newcommand{\tO}{\Omega_{\mathrm{e}}}

\newcommand{\Qb}{Q^\mathrm{bvp}}

\newcommand{\yb}{y_|}
\newcommand{\etab}{\eta_|}

\newcommand{\tf}{p_f}
\newcommand{\tp}{\tilde{p}}

\DeclareMathOperator{\WF}{WF}
\DeclareMathOperator{\supp}{supp}

\DeclareMathOperator{\ufour}{\mathcal{U}^{(4)}}
\DeclareMathOperator{\codim}{codim}

\newcommand{\Rthree}{\mathbb{R}^3}
\newcommand{\ptx}{p(t,x')}

\newcommand{\epslam}{{\partial_{\epsilon_1}\partial_{\epsilon_2}\partial_{\epsilon_3}\partial_{\epsilon_4} \Lambda(f) |_{\epsilon_1 = \epsilon_2 = \epsilon_3 = \ep_4=0}}}

\newcommand{\lambdaFone}{\Lambda_{F^{(1)}}}
\newcommand{\lambdaFtwo}{\Lambda_{F^{(2)}}}

\newcommand{\sq}{\square_c}
\newcommand{\Qbg}{Q_g^{\text{bvp}}}

\newcommand{\Mo}{\mathbb{R} \times \Omega}
\newcommand{\pMo}{\mathbb{R} \times \partial \Omega}

\newcommand{\ziz}{\zeta_0^{(i)}}
\newcommand{\zjz}{\zeta_0^{(j)}}
\newcommand{\zkz}{\zeta_0^{(k)}}
\newcommand{\zlz}{\zeta_0^{(l)}}

\newcommand{\zh}{\hat{\zeta}}
\newcommand{\thett}{({\theta}/{2})}
\newcommand{\costt}{\cos \theta}

\newcommand{\sintw}{\sin \thett}
\newcommand{\costw}{\cos \thett}

\newcommand{\Wset}{\mathbb{W} = \bigcup_{y^-, y^+ \in (0,T) \times \partial \Omega} I(y^-, y^+) \cap \intM}

\newcommand{\ba}{\[\begin{aligned}}
\newcommand{\ea}{\end{aligned}\]}

\title[Nonlinear ultrasound imaging]{An Inverse Boundary Value Problem arising in Nonlinear Acoustics} 
\author{Gunther Uhlmann and Yang Zhang}

\begin{document}
\maketitle

\begin{abstract}
{
We consider an inverse problem arising in nonlinear ultrasound imaging.
The propagation of ultrasound waves is modeled by a quasilinear wave equation. 
We make measurements at the boundary of the medium encoded in the Dirichlet-to-Neumann map, and we show that
these measurements determine the nonlinearity.
}
\end{abstract}

\section{Introduction}
{
Nonlinear ultrasound waves are widely used in medical imaging,
since the propagation of high-intensity ultrasound is 
not adequately modeled by linear acoustic equations, see \cite{humphrey2003non}. 
It has many applications in diagnostic and therapeutic medicine, for example, the probe of tissues, the visualization of blood flow, and the potential application in monitoring patients in the operating room.
For more details and other applications, see \cite{Anvari2015, Aubry2016,Demi2014,Demi2014a,Eyding2020,Fang2019, Gaynor2015, Harrer2003,Hedrick2005,Soldati2014, Szabo2014, Haar2016, Thomas1998,Webb2017}.

We consider  a bounded domain  $\Omega \subset \Rthree$ with smooth boundary.
Let  $(t,x') \in  \mathbb{R} \times \Omega$ and $c(x') > 0$ be the sound speed.
Let $\ptx$ denote the pressure field of the ultrasound waves. 
A model for the pressure field in the medium $\Omega$ with no memory terms is given by (see \cite{Kaltenbacher2021})
\begin{equation}\label{eq_problem}
\begin{aligned}
 \partial_t^2 \ptx  - c^2(x') \Delta \ptx
- F(x', p,\partial_t p, \partial^2_t p) &= 0, & \  & \mbox{in } (0,T) \times \Omega,\\
\ptx &= f, & \ &\mbox{on } (0,T) \times \partial \Omega,\\
p = {\partial_t p} &= 0, & \  & \mbox{on } \{t=0\},
\end{aligned}
\end{equation}
where $f$ is the insonation profile on the boundary and $F$ is the nonlinear term modeling the nonlinear
interaction of the waves.


{
When $F(x', p,\partial_t p, \partial^2_t p) = \beta(x') \partial_t^2 (p^2)$,
this equation is called the Westervelt equation.
In this case, the inverse problem of recovering $\beta$ from some
measurements of the waves is studied in \cite{nonlinearWestervelt, ultra21}. 
The Westervelt equation can be regarded as a lower order approximation to the more complicated physical model (\ref{eq_problem}).  
In \cite{Kaltenbacher2021}, an inverse problem modeled by (\ref{eq_problem}) with a more general  
nonlinear term given by 
\begin{align}\label{def_F}
F(x', p,\partial_t p, \partial^2_t p) = \partial_t(f_W(x',p) \partial_t p)
\end{align}
is considered. 
In this paper, 
we study the inverse boundary value problem of recovering the general nonlinear term of the form (\ref{def_F}) from the boundary measurements $\Lambda_F$, 
where  $\Lambda_F$ is the Dirichlet-to-Neumann (DN) map given by
\[
\Lambda_F f = \partial_\nu p|_{(0, T) \times \partial \Omega},
\]
with $\nu$ as the outer unit normal vector to $\partial \Omega$.
We additionally assume that $f_W(x',p)$ is analytic in $p$ and has the expansion
\[
f_W(x',p)  = \sum_{m=1}^{+\infty} b_{m}(x') p^m
\]
with $b_m(x')$ smooth.}
Then we can rewrite it as
\[
F(x', p,\partial_t p, \partial^2_t p) = \sum_{m=1}^{+\infty} \beta_{m+1}(x') \partial_t^2 (p^{m+1}),
\]
where $\beta_{m+1}(x') = b_m(x')/(m+1)$ is smooth.
More generally, our method works for the case when the nonlinear coefficient $\beta_m$ depends on both $t$ and $x'$, i.e., we consider the nonlinear term
\begin{align}\label{eq_nlterm}
F(x, p,\partial_t p, \partial^2_t p) = \sum_{m=1}^{+\infty} \beta_{m+1}(x) \partial_t^2 (p^{m+1}).
\end{align}
We make the assumption that
\begin{equation}\tag{A}\label{assum_F}
\begin{aligned}
&\sum_{m=1}^{+\infty} \beta_{m+1}(x)  p^{m+1} \text{ is analytic in } p
\text{ and} \text{ there is } m \geq 1 \text{ such that }
\beta_{m+1}(x) \neq 0,
\end{aligned}
\end{equation}
}
\subsection{Main Results.}
We use the notation $x= (t,x') = (x^0, x^1, x^2, x^3) $ and $M = \mathbb{R} \times \Omega$ in some cases for convenience. 
Consider the Riemannian metric
\begin{align}\label{eq_g0}
g_0 = c^{-2}(x') ((\diff x^1)^2 + (\diff x^2)^2 + (\diff x^3)^2)
\end{align}
w.r.t. the wave speed $c(x')>0$ on $\Omega$.
\begin{assumption}\label{assum_Omega}
Suppose the Riemannian manifold $(\Omega, g_0)$ is nontrapping and $\partial \Omega$ is strictly convex w.r.t. $g_0$.
By nontrapping, we mean there exists $T>0$ such that
   \[
    \mathrm{diam}_{g_0}(\Omega) = \sup\{\text{lengths of all geodesics in } (\Omega, g_0)\} < T.
    \]
\end{assumption}
First suppose that each $\beta_{m+1}(x'), m \geq 1$ is independent of $t$ and the sound speed $c(x')>0$ is smooth and known.  
We have the following result.
\begin{thm}\label{thm1}
Let $(\Omega, g_0)$ satisfy Assumption \ref{assum_Omega}.
Consider the nonlinear wave equation
\[
 \partial_t^2 p - c^2(x')\Delta p - F(x', p,\partial_t p, \partial^2_t p) = 0,
\]
where the nonlinear term depending on $x'$ smoothly has the expansion
\begin{align*}
F(x', p,\partial_t p, \partial^2_t p) = \sum_{m=1}^{+\infty} \beta_{m+1}(x') \partial_t^2 (p^{m+1})
\end{align*}
and satisfies the assumption (\ref{assum_F}) for each $x' \in \Omega$.
Let $p^{(k)}$ be the solutions corresponding to
$\Fk(x', \pk,\partial_t \pk, \partial^2_t \pk)$, for $k=1,2$.
If the Dirichlet-to-Neumann maps satisfy
\[
\lambdaFone(f) = \lambdaFtwo(f)
\] 
for all $f$ in a small neighborhood of the zero functions in $C^6([0,T] \times \partial \Omega)$,
then
$\beta^{(1)}_m(x') = \beta^{(2)}_m(x')$,
for any $m \geq 2$ and $x' \in \Omega$.
\end{thm}
If the wave speed $c(x')$ is unknown, one can recover it from the first order linearization of the DN map, i.e., the DN map for the linear problem, assuming it is smooth, see \cite{Belishev1987}.
For the stable recovery with partial data, see \cite{MR3454376}.
After recovering $c(x')$, one can recover $\beta_m$ for $m\geq 2$ by Theorem \ref{thm1}.
This gives the following result.

\begin{thm}\label{cor1a}
    For $k = 1,2$, suppose the wave speeds $c^{(k)}(x') > 0$ are smooth.
    Let $g_0^{(k)}$ be the Riemannian metrics corresponding to $c^{(k)}(x')$, see (\ref{eq_g0}). 
    Suppose $(\Omega, g^{(k)}_0)$ satisfy Assumption \ref{assum_Omega}.
    Consider the  nonlinear wave equations 
    \[
    \partial_t^2 \pk - (c^{(k)}(x'))^2\Delta \pk - \Fk(x', \pk,\partial_t \pk, \partial^2_t \pk) = 0, \quad k = 1,2,
    \]
    where the nonlinear terms depending on $x'$ smoothly have the expansion
    \begin{align*}
    \Fk(x', \pk,\partial_t \pk, \partial^2_t \pk) = \sum_{m=1}^{+\infty} \beta_{m+1}(x') \partial_t^2 ((\pk)^{m+1})
    \end{align*}
    and satisfy the assumption in (\ref{assum_F}) for each $x' \in \Omega$.
    If the Dirichlet-to-Neumann maps satisfy
    \[
    \lambdaFone(f) = \lambdaFtwo(f)
    \] 
    for all $f$ in a small neighborhood of the zero functions in $C^6([0,T] \times \partial \Omega)$,
    then $c^{(1)}(x') = c^{(2)}(x') $ and 
    $\beta^{(1)}_m(x') = \beta^{(2)}_m(x')$,
    for any $m \geq 2$ and $x' \in \Omega$.
\end{thm}

The result in Theorem \ref{thm1} can be regarded as an example of a more general setting. 
Recall $M = \mathbb{R} \times \Omega$ and let $\intM$ be the interior of $M$.
The linear part of the equation in (\ref{eq_problem}) corresponds to the Lorentzian metric
\begin{align*}
g = - \diff t^2  + g_0 = - \diff t^2 + c^{-2}(x') (\diff x')^2 
\end{align*}
and one denote
\[
 \sq \ptx \equiv  \partial_t^2 \ptx  - c^2(x')\Delta \ptx.
\]
Strictly speaking, the operator $\sq$ is not the Laplace-Beltrami operator $\square_g$ 
but it has the same principal symbol as $\square_g$.
More generally, we can assume the smooth speed $c(t,x')>0$ depends on $t$ as well. 

Note that $(M,g)$ is a globally hyperbolic Lorentzian manifold with timelike boundary $\partial M = \mathbb{R} \times \partial \Omega$.
Additionally, we assume $\partial M$ is null-convex.
Here $\partial M$ is null-convex if for any null vector $v \in T_p \partial M$ one has
\begin{eqnarray}\label{def_nconvex}
\kappa(v,v) = g (\nabla_v \nu, v) \geq 0,
\end{eqnarray}
where we denote by $\nu$ the outward pointing unit normal vector field on $\partial M$.
This is true especially when $\partial \Omega$ is convex w.r.t $g_{0}$.
In the following, we consider a globally hyperbolic Lorentzian manifold $(M,g)$ with timelike and null-convex boundary. 

We first introduce some definitions to state the main results.
A smooth path $\mu:(a,b) \rightarrow M$ is timelike if $g(\dot{\mu}(s), \dot{\mu}(s)) <0 $ for any $s \in (a,b)$.
It is causal if $g(\dot{\mu}(s), \dot{\mu}(s)) \leq 0 $ and $\dot{\mu}(s) \neq 0$ for any $s \in (a,b)$.
For $p, q \in M$, we denote by $p < q$ (or $p \ll q$) if $p \neq q$ and there is a future pointing casual (or timelike) curve from $p$ to $q$.
We denote by $p \leq q$ if either $p = q$ or $p<q$.
The chronological future of $p$ is the set $I^+(p) = \{q \in M: \ p \ll q\}$
and the causal future of $p$ is the  set $J^+(p) = \{q \in M: \ p \leq q\}$.
Similarly we can define the chronological past $I^-(p)$ and the causal past $J^-(p)$.
For convenience, we use the notation $J(p,q)= J^+(p) \cap J^-(q)$ to denote the diamond set $J^+(p) \cap J^-(q)$ and
$I(p,q)$ to denote the set $I^+(p) \cap I^-(q)$, see Figure \ref{picref}.
\begin{figure}[h]
    \begin{subfigure}[t]{0.5\textwidth}
    \centering
    \includegraphics[height=0.7\textwidth]{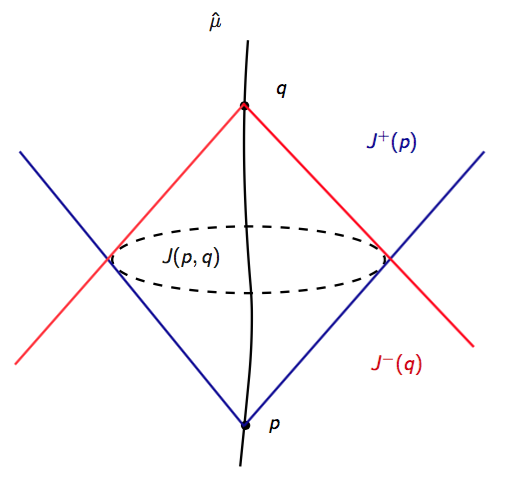}
    \end{subfigure}%
    \begin{subfigure}[t]{0.5\textwidth}
    \centering
    \includegraphics[height=0.75\textwidth]{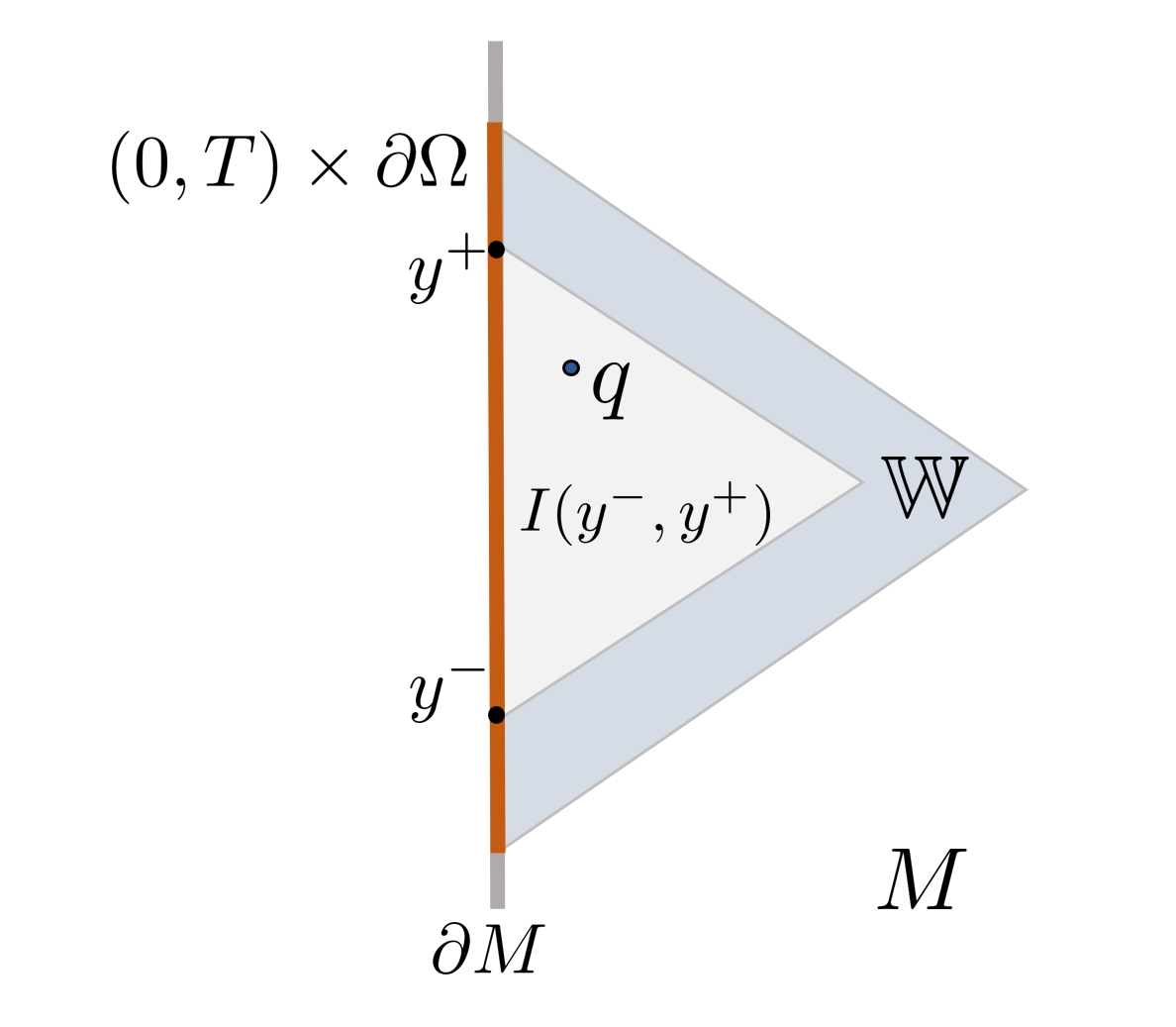}
    \end{subfigure}%
\caption{Left: the causal future $J^+(p)$, causal past $J^-(q)$, and the diamond set $J(p,q)$. Right: the set $\mathbb{W}$.}
\label{picref}
\end{figure}
We consider the recovery of the nonlinear coefficients in a suitable larger set
\[
\Wset.
\]
{
\begin{thm}\label{thm2}
Let $\Omega$ be a bounded domain with smooth boundary in $\mathbb{R}^3$. 
Let $M = \mathbb{R} \times \Omega$ and $g =  - \diff t^2 + c^{-2}(t,x')(\diff x')^2$ with $c(t,x') > 0$ for any $x'\in \Omega$. 
Suppose the globally hyperbolic Lorentzian manifold $(M,g)$ has timelike and null-convex boundary. 
We define the wave operator 
\[
\sq \ptx \equiv \partial_t^2 \ptx  - c^2(t,x')\Delta \ptx. 
\]
Consider the nonlinear wave equation
\[
\sq p^{(k)} - \Fk(x, \pk,\partial_t \pk,\partial^2_t \pk) = 0, \quad  k=1,2,
\]
where the nonlinear terms depending on $x$ smoothly have the expansions
\begin{align*}
\Fk(x, \pk,\partial_t \pk, \partial^2_t \pk) = \sum_{m=1}^{+\infty} \beta^{(k)}_{m+1}(x) \partial_t^2 ((p^{(k)})^{m+1})
\end{align*}
and satisfy the assumption in (\ref{assum_F}) for each $x \in \mathbb{W}$.
If the Dirichlet-to-Neumann maps  satisfy
\[
\Lambda^{(1)}_{F^{(1)}}(f) = \Lambda^{(2)}_{F^{(2)}}(f)
\] 
for all functions $f$ in a small neighborhood of the zero functions in $C^6([0,T] \times \partial \Omega)$,
then
$
\beta^{(1)}_m = \beta^{(2)}_m,$
{for any } $m \geq 2$ and $x \in \mathbb{W}$.
\end{thm}
}
{
When $c(t,x')$ is unknown, the problem of recovering $c(t,x')$ from the DN map for the linear problem is still open in general.
This theorem shows the unique recovery of the nonlinear term from the DN map under our assumptions, when the sound speed $c(t,x')$ is known.
We remark that the same result holds if we consider the metric $g = -\diff t^2 + \kappa(t,x')$ and replace $\square_c$ by 
the operator $\square = \square_g + A(x, D_x)$, where $\kappa(t,x')$ is family of Riemannian metrics on $\Omega$ smoothly depending on $t$ and 
$A$ is a first order linear differential operator.}
The inverse problems of recovering the metric and the nonlinear term for a semilinear wave equation  were originated in \cite{Kurylev2018} in a globally hyperbolic Lorentzian manifold without boundary. 
Starting with \cite{Kurylev2018, Kurylev2014a},
there are many works studying inverse problems for nonlinear hyperbolic equations, see
\cite{Barreto2021,Barreto2020, Chen2019,Chen2020,Hoop2019,Hoop2019a,Uhlig2020,Feizmohammadi2019,Hintz2020, Kurylev2014,Balehowsky2020,Lai2021,Lassas2017,Tzou2021,Uhlmann2020,Hintz2021,ultra21,Uhlmann2019}.
For an overview of the recent progress, see \cite{lassas2018inverse,Uhlmann2021}.

The plan of this paper is as follows.
In Section \ref{Sec_well}, we establish the well-posedness of the boundary value problems (\ref{eq_problem}) for small boundary data.
In Section \ref{sec_prelim}, we present some preliminaries for Lorentzian geometry as well as microlocal analysis, construct the distorted planes waves, and derive the asymptotic expansion.
We prove Proposition \ref{pp_thm1}, which implies that Theorem \ref{thm1} is a special case of Theorem \ref{thm2}. Therefore, in the rest of the paper the aim is to prove Theorem \ref{thm2} using nonlinear interaction of distorted plane waves. 
In Section \ref{Sec_threewaves} and \ref{Sec_fourwaves}, we analyze the singularities produced by the nonlinear interaction of three and four distorted plane waves.
Based on these results, 
we recover the lower order nonlinear coefficients from the fourth order linearization of the DN map in Section \ref{sec_lower}.
The recovery is based on a special construction of lightlike covectors at each $q\in \mathbb{W}$.
In Section \ref{sec_Piriou} and
Section \ref{sec_N}, we recover other nonlinear coefficients from the higher order linearization of the DN map, using Piriou conormal distributions.

\section{Local well-posedness}\label{Sec_well}
The well-posedness of the boundary value problem (\ref{eq_problem}) with small boundary data $f$ can be established by the same arguments as in \cite{Hintz2020}, see also
\cite{ultra21} and \cite{Uhlmann2021a}.
For completeness we present the proof below.

Let $T >0$ be fixed and let $m \geq 5$.
Suppose $f \in C^{m+1} ([0,T] \times \partial \Omega)$ satisfies $\|f\|_{C^{m+1} ([0,T] \times \partial \Omega)} \leq \epsilon_0$, with small positive number $\epsilon_0$ to be specified later.
Then there exists a function $\tf \in C^{m+1} ([0,T] \times \Omega)$ such that $ \tf|_{\partial M} = f$ and
\[
\|\tf\|_{C^{m+1} ([0,T] \times \Omega)} \leq\|f\|_{C^{m+1} ([0,T] \times \partial \Omega)} .
\]
Let $\tp = p -\tf $.
We rewrite the nonlinear term as
\begin{align}\label{eq_F}
F(x, p, \partial_t p, \partial^2_t p)
& = \sum_{j = 1}^{+\infty} \beta_{j+1}(x) \partial_t^2(p^{j+1}) \\
&  = (\sum_{j = 1}^{+
    \infty} (j+1)\beta_{j+1}(x)  p^{j}) p_{tt} 
+ (\sum_{j = 1}^\infty (j+1)j  \beta_{j+1}(x) p^{j-1}) p_{t}p_t \nonumber\\
& \equiv q_1(x,p)pp_{tt} + q_2(x,p) p_t^2. \nonumber
\end{align}
By the assumption (\ref{assum_F}), the functions $q_1, q_2$ 
are smooth functions over $M \times \mathbb{R}$. 
Then $\tp$ must solve the system
\begin{equation}\label{hmNBC}
\begin{cases}
\sq \tp = - \sq \tf  + F(x, \tp + \tf,  \partial_t(\tp + \tf), \partial^2_t(\tp + \tf)) , &\mbox{on } M,\\
\tp = 0, &\mbox{on } \partial M,\\
\tp = 0, &\mbox{for } t <0.
\end{cases}
\end{equation}
For $R>0$, we define $Z^m(R,T)$ as the set containing all functions $v$ such that
\[
v \in \bigcap_{k=0}^{m} W^{k, \infty}([0,T]; H^{m-k}(\Omega)),
\quad \|v\|^2_{Z^m} = \sup_{t \in [0, T]} \sum_{k=0}^m \|\partial_t^k v(t)\|^2_{H^{m-k}(\Omega)} \leq R^2.
\]
We abuse the notation $C$ to denote different constants that depends on $m, M, T$.
One can show the following claim by Sobolev Embedding Theorem, see also \cite[Claim 3]{Uhlmann2021a} and its proof.
\begin{claim}\label{normineq}
Suppose $u \in Z^m(R, T)$.
Then $\|u\|_\zmm \leq \|u\|_\zm$ and $\nabla^j_g u \in Z^{m-1}(R, T)$,
$j = 1, \dots, 4$. Moreover, we have the following estimates.
\begin{enumerate}[(1)]
    \item If $v \in Z^m(R', T)$, then $\|uv\|_\zm \leq C \|u\|_\zm \|v\|_\zm$.
    \item If $v \in \zmm(R', T)$, then $\|uv\|_\zmm \leq C \|u\|_\zm \|v\|_\zmm$.
    \item If $q(x,u, w) \in C^m(M \times \mathbb{C} \times \mathbb{C}^n)$,
    then $\|q(x, u, \diff u)\|_\zm \leq C \|q\|_{C^m(M \times \mathbb{C} \times \mathbb{C}^n)} (\sum_{l=0}^{m} \|u\|^l_\zm)$.
\end{enumerate}
\end{claim}
For $v \in Z^m(\rho_0, T)$ with $\rho_0$ to be specified later, we consider the linearized problem
\begin{equation*}
\begin{cases}
\sq \tp = - \sq \tf  - F(x, v + \tf, \partial_t (v + \tf), \partial^2_t (v + \tf)) , &\mbox{on } M,\\
\tp = 0, &\mbox{on } \partial M,\\
\tp = 0, &\mbox{for } t <0,
\end{cases}
\end{equation*}
and we define the solution operator $\mathcal{J}$ which maps $v$ to the solution $\tilde{u}$.
By Claim \ref{normineq} and (\ref{eq_F}),
we have
\begin{align*}
&\| - \sq \tf  - F(x, v + \tf, \partial_t(v + \tf), \partial^2_t(v + \tf)) \|_{Z^{m-2}} \\
\leq &  \| - \sq \tf \|_{C^{m-1}([0,T]\times \Omega)} +  C\sum_{k=1}^2\|q_k(x,v + \tf)\|_\zm \|v + \tf\|_\zm \|v+\tf \|_\zm \\
\leq & C( \ep_0 + (1 + (\rho_0 + \ep_0) + \ldots +  (\rho_0 + \ep_0)^m) (\rho_0 + \ep_0)^2 ).
\end{align*}
According to \cite[Theorem 3.1]{Dafermos1985},
the linearized problem has a unique solution
\[
\tp \in \bigcap_{k=0}^{m} C^{k}([0,T]; H^{m-k}(\Omega))
\]
such that
\[
 \| \tp\|_{Z^m} \leq C(\ep_0 + (1 + (\rho_0 + \ep_0) + \ldots +  (\rho_0 + \ep_0)^m) (\rho_0 + \ep_0)^2 )e^{KT},
\]
where $C, K$ are positive constants.
If we assume $\rho_0$ and $\epsilon_0$ are small enough, then the above inequality implies that
\[
\| \tp\|_{Z^m} \leq C(\ep_0 + (\rho_0 + \ep_0)^2 )e^{KT}.
\]
For any $\rho_0$ satisfying $\rho_0 < 1/({2C e^{KT}})$,
we can choose $\ep_0 = {\rho_0}/({8C e^{KT}}) $  such that
\begin{equation}\label{eps}
C(\ep_0 + (\rho_0 + \ep_0)^2 )e^{KT} < \rho_0.
\end{equation}
In this case, we have  $\mathcal{J}$ maps $Z^m(\rho_0, T)$ to itself.

In the following we show that $\mathcal{J}$ is a contraction map if $\rho_0$ is small enough.
It follows that the boundary value problem (\ref{hmNBC}) has a unique solution $\tilde{u} \in Z^m(\rho_0, T)$ as a fixed point of $\mathcal{J}$.
Indeed, for $\tp_j = \mathcal{J}(v_j)$ with $v_j \in Z^m(\rho_0, T)$, we have that $\tp_2 - \tp_1$ satisfies
\begin{align*}
&\sq (\tp_2 - \tp_1) \\
=& F(x,v_2 + \tf, \partial_t(v_2 + \tf), \partial^2_t(v_2 + \tf))  -F(x,v_1 + \tf, \partial_t(v_1 + \tf), \partial^2_t(v_1 + \tf))\\
= & (q_1(x,v_2 +\tf)(v_2 +\tf) -q_1(x,v_1 +\tf)(v_1 + \tf)\partial_t^2(v_2 + \tf) \\
&+ q_1(x,v_1 +\tf)(v_1 +\tf)\partial^2_{t}(v_2-v_1)\\
&+ (q_2(x,v_2 +\tf) -q_2(x,v_1 +\tf))\partial_t(v_2 +v_1+ 2\tf)\partial_t(v_2 - v_1).
\end{align*}
We denote the right-hand side by $\mathcal{I}$ and using Claim \ref{normineq} for each term above,
we have 
\begin{align*}
\|\mathcal{I}\|_{Z^{m-2}}
&\leq C'  \|v_2 - v_1\|_\zm (\rho_0 + \ep_0),
\end{align*}
where $\rho_0, \ep_0$ are chosen to be small enough.
By \cite[Theorem 3.1]{Dafermos1985} and (\ref{eps}), one obtains
\begin{align*}
&\| \tp_2 - \tp_1 \|_\zm 
&\leq CC' \|v_2 - v_1\|_\zm (\rho_0 + \ep_0) e^{KT} < CC'{{ e^{KT} }} (1 + 1/(8Ce^{KT}))\rho_0 \|v_2 - v_1\|_\zm.
\end{align*}
Thus, if we choose $\rho \leq \frac{1}{CC'e^{KT}(1 + 1/(8Ce^{KT}))} $, then
\[
\| \mathcal{J}(v_2 -v_1)\|_\zm < \|v_2 - v_1\|_\zm
\]
shows that $\mathcal{J}$ is a contraction.
This proves that there exists a unique solution $\tilde{u}$ to the problem (\ref{hmNBC}).
Furthermore, by \cite[Theorem 3.1]{Dafermos1985} this solution
satisfies the estimates $\|\tp\|_\zm \leq 8C e^{KT} \ep_0.$
Therefore, we prove the following proposition.
\begin{pp}
Let $f \in C^{m+1} ([0,T] \times \partial \Omega)$ with $m \geq 5$.
Suppose $f = \partial_t f = 0$ at $t=0$.
Then there exists small positive $\ep_0$ such that for any
$\|f\|_{C^{m+1} ([0,T] \times \partial \Omega)} \leq \epsilon_0$, we can find a unique solution
\[
p \in \bigcap_{k=0}^m C^k([0, T]; H_{m-k}(\Omega))
\]
to the boundary value problem (\ref{eq_problem}), which satisfies the estimate
\[
\|{p}\|_\zm \leq C \|f\|_{C^{m+1}([0,T]\times \partial \Omega)}
\]
for some $C>0$ independent of $f$.
\end{pp}


\section{Preliminaries}\label{sec_prelim}

\subsection{Lorentzian manifolds}
Recall $(M,g)$ is globally hyperbolic with timelike and null-convex boundary, where $M = \mathbb{R} \times \Omega$.
As in \cite{Hintz2020}, we extend $(M,g)$ smoothly to a slightly larger globally hyperbolic Lorentzian manifold $(\tM, g_e)$ without boundary, 
where $\tM  = \mathbb{R} \times \tO$ 
such that $\Omega$ is contained in the interior of the open set $\tO$.
Let 
\[V  = (0,T) \times \tO \setminus \Omega\] 
be the observation set. 
See also \cite[Section 7]{Uhlmann2021a} for more details about the extension. 
In the following, we abuse the notation and do not distinguish $g$ with $g_e$ if there is no confusion caused.


We recall some notations and preliminaries in \cite{Kurylev2018}.
For $\eta \in T_p^*\tM$, the corresponding vector of $\eta$  is denoted by $ \eta^\# \in T_p \tM$.
The corresponding covector of a vector $\xi \in T_p \tM$ is denoted by $ \xi^b \in T^*_p \tM$.
We denote by 
\[
L_p \tM = \{\zeta \in T_p \tM \setminus 0: \  g(\zeta, \zeta) = 0\}
\]
the set of light-like vectors at $p \in \tM$ and similarly by $L^*_p \tM$ the set of light-like covectors.
The sets of future (or past) light-like vectors are denoted by $L^+_p \tM$ (or $L^-_p \tM$), and those of future (or past) light-like covectors are denoted by $L^{*,+}_p \tM$ (or $L^{*,-}_p \tM$).

The characteristic set $\Char(\sq)$ is the set $b^{-1}(0) \subset T^*\tM$, where 
$
b(x, \zeta) =g^{ij}\zeta_i \zeta_j
$
is the principal symbol. 
It is also the set of light-like covectors with respect to $g$.  
We denote by $\Theta_{x, \zeta}$ the null bicharacteristic of $\sq$ that contains $(x, \zeta) \in L^*\tM$,
which is defined as the integral curve of the Hamiltonian vector field $H_b$. 
Then a covector $(y, \eta) \in \Theta_{x, \zeta}$ if and only if there is a light-like geodesic $\gamma_{x, \zeta^\#}$ such that 
\[
(y, \eta) = (\gamma_{x, \zeta^\#}(s), (\dot{\gamma}_{x, \zeta^\#}(s))^b), \ \text{ for } s \in \mathbb{R}.
\]
Here we denote by $\gamma_{x, \zeta^\#}$ the unique null geodesic starting from $x$ in the direction $\zeta^\#$. 

The time separation function $\tau(x,y) \in [0, \infty)$ between two points $x < y$ in  $\tM$
is the supremum of the lengths \[
L(\alpha) =  \int_0^1 \sqrt{-g(\dot{\alpha}(s), \dot{\alpha}(s))} ds
\] of 
the piecewise smooth causal paths $\alpha: [0,1] \rightarrow \tM$ from $x$ to $y$. 
If $x<y$ is not true, we define $\tau(x,y) = 0$. 
Note that $\tau(x,y)$ satisfies the reverse triangle inequality
\[
\tau(x,y) +\tau(y,z) \leq \tau(x,z), \text{ where } x \leq y \leq z.
\]
For $(x,v) \in L^+\tM$, recall the cut locus function
\[
\rho(x,v) = \sup \{ s\in [0, \mathcal{T}(x,v)]:\ \tau(x, \gamma_{x,v}(s)) = 0 \},
\]
where $\mathcal{T}(x,v)$ is the maximal time such that $\gamma_{x,v}(s)$ is defined.
The cut locus function for past lightlike vector $(x,w) \in L^-\tM$ is defined dually with opposite time orientation, i.e.,
\[
\rho(x,w) = \inf \{ s\in [\mathcal{T}(x,w),0]:\ \tau( \gamma_{x,w}(s), x) = 0 \}.
\]
For convenience, we abuse the notation $\rho(x, \zeta)$ to denote $\rho(x, \zeta^\#)$ if $\zeta \in L^{*,\pm}\tM$.
By \cite[Theorem 9.15]{Beem2017}, the first cut point $\gamma_{x,v}(\rho(x,v))$ is either the first conjugate point or the first point on $\gamma_{x,v}$ where there is another different geodesic segment connecting $x$ and  $\gamma_{x,v}(\rho(x,v))$. 

In particular, 	when $g = -\diff t^2 + g_0$, we have the following proposition, which implies Theorem \ref{thm1} is the result of Theorem \ref{thm2}. 
\begin{pp}\label{pp_thm1}
Let $(\Omega, g_0)$ satisfy the assumption (\ref{assum_Omega}) and $g = -\diff t^2 + g_0$, see (\ref{eq_g0}) for the definition of $g_0$. 
For any $x'_0 \in \Omega$, one can find a point $q\in \mathbb{W}$ with $q = (t_q, x'_0)$ for some $t_q \in (0, T)$. 
\end{pp}
\begin{proof}
For each $x'_0 in \Omega$, 
we consider a unit-speed geodesic $\lambda(s)$ in $\Omega$ such that $x'_0 = \lambda(s_0)$ and there exists
\[
s_1  = \sup \{s < s_0: \lambda(s) \in \partial \Omega\}.
\]  
We write $p' = \lambda(s_1)$. 
With (\ref{assum_Omega}), we can assume $s_0 -s_1 < {T}/{2}$. 
There exists $\ep >0$ such that $2(s_0-s_1) + \ep < T$. 
First we consider 
\[
\gamma_1(s) = (s + \ep, \lambda(s+s_1)).
\]
It follows that $\gamma_1(s)$ is a null geodesic in $M$ with $\gamma_1(0) = (\ep, p') \in (0,T) \times \partial \Omega$. 
We set $q = \gamma_1(s_0- s_1) = (s_0-s_1 + \ep, x'_0)$. 
Note that $t_q = s_0 - s_1 + \ep \in (0, T)$. 
Next, we consider 
\[
\gamma_2(s) = (s + s_0-s_1 + \ep, \lambda(s_0 - s)).
\]
Similarly, $\gamma_2(s)$ is a null geodesic in $M$ with $\gamma_2(0) = (s_0 - s_1 + \ep, x'_0) = q$ and $\gamma_2(s_0 - s_1) = (2(s_0 -s_1) + \ep, p') \in (0, T) \times \partial \Omega$.
Therefore, we have $q \in J(\gamma_1(0), \gamma_2(s_0-s_1))$, which implies $q\in \mathbb{W}$. 
\end{proof}

\subsection{Distributions}
Suppose $\Lambda$ is a conic Lagrangian submanifold in $T^*\tM$ away from the zero section. 
We denote by $\Ical^\mu(\Lambda)$ the set of Lagrangian distributions in $\tM$ associated with $\Lambda$ of order $\mu$. 
In local coordinates, a Lagrangian distribution can be written as an oscillatory integral and we regard its principal symbol, 
which is invariantly defined on $\Lambda$ with values in the half density bundle tensored with the Maslov bundle, as a function in the cotangent bundle. 
If $\Lambda$ is a conormal bundle of a submanifold $K$ of $\tM$, i.e. $\Lambda = N^*K$, then such distributions are also called conormal distributions. 
The space of distributions in $\tM$ associated with two cleanly intersecting conic Lagrangian manifolds $\Lambda_0, \Lambda_1 \subset T^*\tM \setminus 0$ is denoted by $\Ical^{p,l}(\Lambda_0, \Lambda_1)$. 
If $u \in \Ical^{p,l}(\Lambda_0, \Lambda_1)$, then one has $\wfset{(u)} \subset \Lambda_0 \cup \Lambda_1$ and 
\[
u \in \Ical^{p+l}(\Lambda_0 \setminus \Lambda_1), \quad  u \in \Ical^{p}(\Lambda_1 \setminus \Lambda_0)
\]
away from their intersection $\Lambda_0 \cap \Lambda_1$. The principal symbol of $u$ on $\Lambda_0$  and  $\Lambda_1$ can be defined accordingly and they satisfy some compatible conditions on the intersection. 

For more detailed introduction to Lagrangian distributions and paired Lagrangian distributions, see \cite[Section 3.2]{Kurylev2018} and \cite[Section 2.2]{Lassas2018}. 
The main reference are \cite{MR2304165, Hoermander2009} for conormal and Lagrangian distributions and
\cite{Hoop2015,Greenleaf1990,Greenleaf1993,Melrose1979} for paired Lagrangian distributions. 

{
\subsection{The causal inverse}
On a globally hyperbolic Lorentzian manifold $(M,g)$, the wave operator $\sq$ with the principal symbol $b(x, \zeta) =g^{ij}\zeta_i \zeta_j$ is normally hyperbolic, see \cite[Section 1.5]{Baer2007}.  
It has a unique casual inverse $\sq^{-1}$ according to \cite[Theorem 3.3.1]{Baer2007}. 
By \cite[Proposition 6.6]{Melrose1979}, one can symbolically construct a parametrix $Q_g$, which is the solution operator to the wave equation 
\begin{align*}
\sq v &= f, \quad \text{ on } \tM,\\
v & = 0, \quad \text{ on } \tM \setminus J^+(\supp(f)),
\end{align*}
in the microlocal sense.
It follows that $Q_g \equiv \sq^{-1}$ up to a smoothing operator.
We denote the kernel of $Q_g$ by $q(x, \tilde{x})$ and 
it is a paired Lagrangian distribution
in $\Ical^{-\frac{3}{2}, -\frac{1}{2}} (N^*\text{Diag}, \Lambda_g)$,
where $\text{Diag}$ denotes the diagonal in $M \times M$, $N^*\text{Diag}$ is its conormal bundle, and $\Lambda_g$ is the flow out of 
$N^*\text{Diag} \cap \Char(\sq)$ under the Hamiltonian vector field $H_b$. 
Here we construct the microlocal solution to the equation 
\[
\sq q(x, \tilde{x}) = \delta(x, \tilde{x}) \mod C^\infty(M \times M).
\]
using the proof of \cite[Proposition 6.6]{Melrose1979},
and 
the symbol of $Q_g$ can be found during the construction there.  
In particular, the principal symbol of $Q_g$ along $N^*\text{Diag}$  satisfying 
$
\sigma_p(\delta) = \sigma_p(\sq) \sigma_p(Q_g)
$
is nonvanishing. 
The one along $\Lambda_g \setminus N^*\text{Diag}$ solves the transport equation
\[
\mathcal{L}_{H_b}\sigma_p(Q_g) + i c\sigma_p(Q_g) = 0,
\]
where $\mathcal{L}_{H_b}$ is the Lie action of the Hamiltonian vector field $H_b$ and $c$ is the subprincipal symbol of $\sq$.  
The initial condition 
is given by restricting $\sigma_p(Q_g)|_{N^*\text{Diag}}$ to $\partial \Lambda_g$, 
see (6.7)  and Section 4 in \cite{Melrose1979}. 
Then one can solve the transport equation by integrating along the bicharacteristics. 
This implies 
the solution to the transport equation is nonzero and therefore 
$\sigma_p(Q_g)|_{\Lambda_g}$ is nonvanishing. 
See also \cite{Hoop2015, Baer2007, Greenleaf1993} for more references. 
 
We have the following proposition according to \cite[Proposition 2.1]{Greenleaf1993}, see also \cite[Proposition 2.1]{Lassas2018}.
\begin{pp}
Let $\Lambda$ be a conic Lagrangian submanifold in $T^*M \setminus 0$. 
Suppose $\Lambda$ intersects $\Char(\sq)$ transversally, such that its intersection with each bicharacteristics has finite many times. 
Then 
\[
\tQ_g: \Ical^\mu(\Lambda) \rightarrow \Ical^{p,l}(\Lambda, \Lambda^g),
\]
where $ \Lambda^g$ is the flow out of $\Lambda \cap \Char(\sq)$ under the Hamiltonian flow. 
Moreover, for $(x, \xi) \in \Lambda^g \setminus \Lambda$, we have
\[
\sigma_p(Q_g u)(x, \xi) = \sum \sigma(Q_g)(x, \xi, y_j, \eta_j)\sigma_p(u)(y_j, \eta_j), 
\]
where the summation is over the points $(y_j, \eta_j) \in \Lambda$ that lie on the bicharacteristics from $(x, \xi)$. 
\end{pp}
}

\subsection{Distorted plane waves.}\label{subsec_planewaves}
We review the distorted plane waves constructed in \cite{Kurylev2018}.
Roughly speaking, they are conormal distributions propagating along the fixed null geodesic before the first cut point.

Let $g^+$ be a Riemannian metric on $\tM$ and $L^+ \tM$ be the bundle of future-pointing light-like vectors.
For $(x_0, \xi_0) \in L^+ \tM$ and a small parameter $s_0 >0$, 
we define 
\begin{align*}
\mathcal{W}({x_0, \xi_0, s_0}) &= \{\eta \in L^+_{x_0} \tM: \|\eta - \xi_0\|_{g^+} < s_0 \text{ with } \|\eta \|_{g^+} = \|\xi_0\|_{g^+}\}
\end{align*}
as a neighborhood of $\xi_0$ at the point $x_0$.
We denote by $\gamma_{x_0, \xi_0}(s), \ s \geq 0$ the unique null geodesic starting from $x_0$ with direction $\xi_0$, and we define
\begin{align*}
K({x_0, \xi_0, s_0}) &= \{\gamma_{x_0, \eta}(s) \in \tM: \eta \in \mathcal{W}({x_0, \xi_0, s_0}), s\in (0, \infty) \}
\end{align*}
as the subset of the light cone emanating from $x_0$ by light-like vectors in $\mathcal{W}({x_0, \xi_0, s_0})$. 
As $s_0$ goes to zero, the surface $K({x_0, \xi_0, s_0})$ tends to the geodesic $\gamma_{x_0, \xi_0}(\mathbb{R}_+)$. 
Consider the Lagrangian submanifold
\begin{align*}
\Sigma(x_0, \xi_0, s_0) =\{(x_0, r \eta^b )\in T^*\tM: \eta \in \mathcal{W}({x_0, \xi_0, s_0}), \ r\neq 0 \},
\end{align*}
which is a subset of the conormal bundle $N^*\{x_0\}$. 
We define
\begin{align*}
\Lambda({x_0, \xi_0, s_0})
= &\{(\gamma_{x_0, \eta}(s), r\dot{\gamma}_{x_0, \eta}(s)^b )\in \tM: \\
& \quad \quad \quad \quad \quad 
\eta \in \mathcal{W}({x_0, \xi_0, s_0}), s\in (0, \infty), r \in \mathbb{R}\setminus \{0 \} \}
\end{align*}
as the flow out from $\Char(\sq) \cap \Sigma(x_0, \xi_0, s_0)$ by the Hamiltonian vector field of $\sq$ in the future direction. 
Note that $\Lambda({x_0, \xi_0, s_0})$ is the conormal bundle of $K({x_0, \xi_0, s_0})$ near $\gamma_{x_0, \xi_0}(\mathbb{R}_+)$ before the first cut point of $x_0$.

Now suppose $J = 3$ or $4$ is given.
According to \cite[Lemma 3.1]{Kurylev2018}, 
we can construct distributions 
\[
u_j \in \Ical^\mu(\Lambda(x_j, \xi_j, s_0)) \text{ satisfying } \square_g u_j \in C^\infty(M), \quad j = 1, \ldots, J, 
\]
with nonzero principal symbol along $(\gamma_{x_j, \xi_j}(s), (\dot{\gamma}_{x_j, \xi_j}(s))^b )$ for $s > 0$. 
Note that $u_j \in \mathcal{D}'(\tM)$ has no singularities conormal to $\partial M$. 
Thus its restriction to the submanifold $\partial M$ is well-defined, see \cite[Corollary 8.2.7]{Hoermander2003}.
Let $f_j =  u_j|_{\partial M}$ and $v_j$ solve the boundary value problem
\begin{equation}\label{eq_v1}
\begin{aligned}
\sq v_j &= 0, & \  & \mbox{on } M ,\\
 v_j &= f_j, & \ &\mbox{on } \partial M,\\
v_j &= 0,  & \  &\mbox{for } t <0.
\end{aligned}
\end{equation}
It follows that $v_j  = u_j \mod C^\infty(M)$. 
We consider the boundary value problem (\ref{eq_problem}) for the nonlinear wave equation with the Dirichlet boundary condition 
$
f = \sum_{j=1}^J \ep_j f_j, 
$
and write the solution $p$ to (\ref{eq_problem}) as an asymptotic expansion with respect to $v_j$ later.

For $j=1,\ldots, J$,  
let $(x_j, \xi_j) \in L^+V$ 
be lightlike vectors.
In some cases, we denote this triplet or quadruplet by $(\vec{x}, \vec{\xi}) = (x_j, \xi_j)^J_{j=1}$. 
We  omit the parameters $x_j, \xi_j, s_0$ and use the following notations 
\[
\gamma_j= \gamma_{x_j, \xi_j}, \quad K_j = K({x_j, \xi_j, s_0}), \quad \Sigma_j=\Sigma({x_j, \xi_j, s_0}), \quad \Lambda_j = \Lambda({x_j, \xi_j, s_0}),
\]
if there is no confusion. 
We say $(x_j, \xi_j)_{j=1}^J$ are causally independent if
\begin{align}\label{assump_xj}
x_j \notin J^+(x_k), 
\quad \text{ for } j \neq k.
\end{align}
Note the null geodesic $\gamma_{x_j, \xi_j}(s)$ starting from $x_j \in V$ could never intersect $M$ or could enter $M$ more than once. 
Thus, we define
\begin{align}\label{def_bpep}
t_j^0  = \inf\{s > 0 : \  \gamma_{x_j, \xi_j}(s) \in M \}, \quad t_j^b  = \inf\{s > t_j^0 : \  \gamma_{x_j, \xi_j}(s) \in \tM \setminus M \}
\end{align}
as the first time when it enters $M$ and 
the first time when it leaves $M$ from inside, 
if such limits exist. 

We introduce the
definition of the regular intersection of three or
four null geodesics at a point $q$, as in \cite[Definition 3.2]{Kurylev2018}. 
\begin{df}\label{def_inter}
    Let $J = 3$ or $4$. 
    We say the geodesics corresponding to
    $(x_j, \xi_j)_{j=1}^J$ 
    intersect regularly at a point $q$,  if  one has
    \begin{enumerate}[(1)]
        \item there are $0 < s_j < \rho(x_j, \xi_j)$ such that $q = \gamma_{x_j, \xi_j}(s_j)$, for $j= 1, \ldots, J$,
        \item the vectors $\dot{\gamma}_{x_j, \xi_j}(s_j), j= 1, \ldots, J$ are linearly independent. 
    \end{enumerate}
\end{df}
It is shown in \cite[Lemma 3.5]{Kurylev2018} that for any $q \in \mathbb{W}$, 
there exists $(z, w) \in L^{+} V$  such that $q = \gamma_{z, w}(s_q)$ with $0 < s_q< \rho(z, w)$. 
Moreover, in any neighborhood of $(z, w)$,  we can find lightlike vectors $(x_j, \xi_j), \ j = 1,2,3,4$ 
such that they are causally independent
as in (\ref{assump_xj}) and the four null geodesics corresponding to $(x_j, \xi_j)$ intersect regularly at $q$.  
To determine the nonlinear terms at fixed point $q \in \mathbb{W}$, in the following we focus on $(x_j, \xi_j)_{j=1}^J$ that are causally independent and intersect regularly at $q$.

For convenience, we introduce the following definition on the intersection of three or four submanifolds as in \cite[Definition 3.1]{Lassas2018}.
\begin{df}\label{df_intersect}
We say four codimension 1 submanifolds $K_1, K_2, K_3, K_4$ intersect 4-transversally if 
\begin{enumerate}[(1)]
    \item $K_i, K_j$ intersect transversally at a codimension $2$ manifold $K_{ij}$, for $i < j$;
    \item $K_i, K_j, K_k$ intersect at a codimension $3$ submanifold $K_{ijk}$, for $i < j < k$;
    \item $K_1, K_2, K_3, K_4$ intersect at a point $q$;
    \item for any two disjoint index subsets $I, J \subset \{1, 2, 3, 4\}$, 
    the intersection of $\cap_{i \in I} K_i$ and $\cap_{j \in J} K_j$ is transversal if not empty. 
\end{enumerate}
If $K_1, K_2, K_3$ satisfy condition (1) and (2), 
then we say they intersect 3-transversally. 
\end{df}
By \cite{Lassas2018}, such $K_1, K_2, K_3, K_4$ intersect at $q$ with linearly independent normal covectors $\zeta^{(j)} \in N_q^*K_j$, $j = 1,2,3,4$. 
These covectors form a basis for the cotangent space $T^*_qM$ such that each $\zeta \in T^*_qM$ has a unique decomposition with respect to them. 
Additionally, if four null geodesics $\gamma_j, j = 1,2,3,4$ intersect regularly at $q$, then we can always construct $K_j$ with small enough $s_0$ such that 
they intersect 4-transversally at $q$. 

For convenience, we introduce the following notations
\[
\Lambda_{ij} = N^*(K_i \cap K_j), \quad \Lambda_{ijk} =  N^*(K_i \cap K_j \cap K_k), \quad \Lambda_q = T^*_q M \setminus 0,
\] 
where $q$ is the intersection point in $\capgamma$. 
We define
\[
\Lambda^{(1)} = \cup_{j=1}^4 \Lambda_j, \quad \Lambda^{(2)} = \cup_{i<j} \Lambda_{ij}, \quad 
\Lambda^{(3)} =  \cup_{i<j<k} \Lambda_{ijk}.
\]
Then we denote the flow out of $\Lambda^{(3)} \cap \Char(\sq)$ under the null bicharacteristics of $\sq$ in $T^*\tM$
by  
\[
\Lambda^{(3), g} = \{(z, \zeta) \in T^*M: \ \exists \ (y, \eta) \in \Lambda^{(3)}\cap \Char(\sq) \text{ such that } (z, \zeta) \in \Theta_{y, \eta} \}.
\]  
The flow out of $\Lambda^{(3)} \cap \Char(\sq)$ under the broken bicharacteristic arcs of $\square_g$ in $T^*M$ is denoted by 
\[
\Lambda^{(3), b} = \{(z, \zeta) \in T^*M: \ \exists \ (y, \eta) \in \Lambda^{(3)} \text{ such that } (z, \zeta) \in \Theta^b_{y, \eta}
\},
\]
see Section \ref{subsec_Qg} for the broken characteristics and the definition of $\Theta^b_{y, \eta}$. 
Let 
\begin{align*}
\Gamma({\vec{x}, \vec{\xi}}, s_0) =  (\Lambda^{(1)} \cup \Lambda^{(2)} \cup \Lambda^{(3)} \cup \Lambda^{(3),b}) \cap T^*M, 
\end{align*}
which depends on the parameter $s_0$ by definition. 
Then we define
\begin{align}\label{def_Gamma}
\unionGamma = \bigcap_{s_0>0}\Gamma({\vec{x}, \vec{\xi}}, s_0)
\end{align}
as the set containing all possible singularities caused by the Hamiltonian flow 
and the interaction of at most three distorted plane waves.

As in \cite{Kurylev2018}, to deal with the complications caused by the cut points, for $j = 1, \ldots, J$, we consider the interaction in the set
\begin{align}\label{def_nxxi}
\nxxi  = M \setminus \bigcup_{j=1}^J J^+(\gamma_{x_j, \xi_j}(\rho(x_j, \xi_j))),
\end{align}
which is the complement of the causal future of the first cut points.
In $\nxxi$, any two of the null geodesics $\gamma_{x_j, \xi_j}(\mathbb{R}_+)$ intersect at most once, by \cite[Lemma 9.13]{Beem2017}.

To deal with the complications caused by the reflection part, for $j = 1, \ldots, J$, we define
\begin{align}\label{def_ntxxi}
\ntxxi  = M \setminus \bigcup_{j=1}^J J^+(\gamma_{x_j, \xi_j}(t_j^b)),
\end{align}
as the complement of the causal future of the point $\gamma_{x_j, \xi_j}(t_j^b)$ in $M$, 
where $\gamma_{x_j, \xi_j}(\mathbb{R}_+)$ leaves $M$ from the inside for the first time. 
Note that in $\nxxi \cap \ntxxi$, each null geodesic $\gamma_{x_j, \xi_j}(\mathbb{R}_+)$ enters $M$ at most once and any two of them intersect at most once.
Let $\intM$ be the interior of $M$ and let $\Sigma(4)$ denote the permutation group of $\{1,2,3,4\}$.
In particular, assuming the four null geodesics intersect regularly at a point in $\intM$, we can show the following lemma.
\begin{lm}\label{lm_intM} 
    Suppose $(x_j, \xi_j)_{j=1}^4$ intersect regularly at $q \in \nxxi \cap \intM$.
    Let $p \in \nxxi \cap \ntxxi$. 
    \begin{itemize}
        \item[(a)] If $(p, \zeta) \in \Lambda^{(2)} \cup \Lambda^{(3)}$ for arbitrarily small $s_0 >0$, then $p \in \intM$. 
        \item[(b)] For fixed $(i,j,k,l) \in \Sigma(4)$, if $(p, \zeta) \in \Lambda_i \cap \Lambda_{jkl}^b$
        with arbitrarily small $s_0 >0$, then $p \in \intM$. 
    \end{itemize}
    
\end{lm}
\begin{proof}
    We prove by contradiction. 
    For (a),
    assume $p \in \tM \setminus \intM$. 
    Then we can find $K_i, K_j$ such that $p$ is in their intersection. 
    If $p \notin \gamma_{x_i, \xi_i}(\mathbb{R}_+) \cap \gamma_{x_j, \xi_j}(\mathbb{R}_+)$, then by choosing small enough $s_0$ one has $p \notin K_i \cap K_j$. 
    This implies $p = \gamma_{x_i, \xi_i}(s_i) =  \gamma_{x_j, \xi_j}(s_j)$ for some $s_i, s_j>0$. 
    In $\ntxxi$, we must have $s_i = t_i^0$ and $s_j = t_j^0$.    
    This contradicts with $p, q \in \nxxi$ by \cite[Lemma 9.13]{Beem2017}. 
    
    For (b), similarly we have $p \in \gamma_{x_i,\xi_i}(\mathbb{R}_+)$ otherwise $p \notin K_i$ for sufficiently small $s_0 > 0$. 
    Since $p \in \ntxxi$, we have $p =\gamma_{x_i, \xi_i}(t_i^0)$, if (b) is not true. 
    It follows from $(p, \zeta) \in \Lambda_{jkl}^b$ that
    there exists $(p_0, \zeta^0) \in \Lambda_{jkl}$ such that $(p, \zeta) \in \Theta^b_{p_0, \zeta^0}$. 
    With $p \in \nxxi$, we must have $p_0 \in \nxxi$ since $p \geq p_0$. 
    Assume $p_0 \notin  \gamma_{x_j,\xi_j}(\mathbb{R}_+) \cap \gamma_{x_k,\xi_k}(\mathbb{R}_+) \cap \gamma_{x_l,\xi_l}(\mathbb{R}_+)$.
    Without loss of generosity, we can assume  $p_0 \notin  \gamma_{x_j,\xi_j}(\mathbb{R}_+)$. 
    Then there exists $s_0>0$ small enough such that $p_0 \notin K_j$, which contradicts with $(p_0, \zeta^0)\in \Lambda_{jkl}$. 
    It follows that $p_0 = q$, since $q$ is the only intersection point in $\nxxi$ by \cite[Lemma 9.13]{Beem2017}.
    Then one has $\gamma_{x_i, \xi_i}(t_i^0) = p > q$, which is impossible. 
\end{proof}
\subsection{The solution operator $\Qb_g$}\label{subsec_Qg}
In this subsection, we present how the singularities propagate after applying $\Qb_g$, i.e., the solution operator to the boundary value problem \begin{equation}\label{bvp_qg}
\begin{aligned}
\sq w &= h, & \  & \mbox{on } \Mo ,\\
w &= 0 , & \  & \mbox{on } \pMo,\\
w &= 0, & \  & \mbox{for } t <0.
\end{aligned}
\end{equation}
The similar analysis is in \cite[Section 3.4]{Uhlmann2021a}. 

We denote the outward (+) and inward (-) pointing tangent bundles by 
\begin{equation}\label{def_tbundle}
\TMpm = \{(x, v) \in \partial TM: \ \pm g(v, n)>0 \},
\end{equation}
where $n$ is the outward pointing unit normal of $\partial M$.
For convenience, we also introduce the notation 
\begin{equation}\label{def_Lbundle}
\LcMpm =\{(z, \zeta)\in L^* M \text{ such that } (z, \zeta^\#)\in \TMpm \}
\end{equation}
to denote the lightlike covectors that are outward or inward pointing on the boundary.

First we recall some definitions and notations in 
\cite{Melrose1978,Melrose1982,Vasy2008,MR2304165,melrose1992atiyah} (see also \cite{LEBEAU1997,Taylor1975}). For a smooth manifold $M$ with boundary $\partial M$, 
let $\bTM$ be the compressed cotangent bundle, see \cite[Lemma 2.3]{melrose1992atiyah}. 
There is a natural map \[
\pi_b: T^*M \rightarrow \bTM
\] 
satisfying that 
it is the identity map in $T^*\intM$ and for $p \in \partial M$ it has the kernel $N^*_p(\partial M)$ and the range identified with $T^*_p(\partial M)$. 
Suppose locally $M$ is given by $(\bx, x^n)$ with $x^n \geq 0$ and the one forms are given by $ \sum \bxi \diff \bx + \xi_n \diff x^n$. 
Then in local coordinates, $\pi_b$ takes the from 
\[
\pi_b(\bx,x^n,\bxi,\xi_n)  = (\bx,x^n,  \bxi, x^n\xi_n).
\]
The compressed cotangent bundle $\bTM$ can be decomposed into the elliptic, glancing, and hyperbolic sets by 
\begin{align*}
&\mathcal{E} = \{\mu \in \bTM \setminus 0,\  \pi_b^{-1}(\mu) \cap \Char(\square_g) = \emptyset \},\\
&\mathcal{G} = \{\mu \in \bTM \setminus 0,\  \text{Card}(\pi_b^{-1}(\mu) \cap \Char(\square_g)) = 1 \},\\
&\mathcal{H} = \{\mu \in \bTM \setminus 0,\  \text{Card}(\pi_b^{-1}(\mu) \cap \Char(\square_g)) = 2 \}.
\end{align*}
Recall $\Char(\sq)$ is the characteristic set of $\sq$ and 
we define its image 
\[
\compchar = \pi_b{(\Char(\sq))}
\] as the compressed bicharacteristic set. 
Obviously the set $\compchar$ is a union of the glancing set $\mathcal{G}$ and the hyperbolic set $\mathcal{H}$.
If $\mu \in T^* \intM$ and $\mu \in \Char{(\sq)}$, then $\mu$ is in $\mathcal{G}$.
Let $\gsetint$ be the subset of $\mathcal{G}$ containing such $\mu$.    
Note  $\gsetint$ can be identified with $ T^* \intM$. 
If $\mu \in T^* \intM$ and $\mu \notin \Char({\sq)}$, then $\mu$ is in the elliptic set $\mathcal{E}$. 

The glancing set on the boundary $\mathcal{G} \setminus \gsetint$ is the set of all points $\mu \in \compchar$ such that $\pi_b^{-1}(\mu) \in T^*(\partial M)$ and the Hamilton vector field $H_p$ of $\sq$ is tangent to $\partial M$ at $\pi_b^{-1}(\mu) $. 
We define $\mathcal{G}^k$ as the  subset where $H_p$ is tangent to $\partial M$ with the order less than $k$, for $k \geq 2$, see \cite[(3.2)]{Melrose1978} and \cite[Definiton 24.3.2]{MR2304165}. 
Note that $\mathcal{G} = \gsetint \cup \mathcal{G}^2$. 
In particular, we denote by $\mathcal{G}^\infty$ the subset with infinite order of bicharacteristic tangency. 
The subset $\mathcal{G}^2 \setminus \mathcal{G}^3$ is the union of the diffractive part $\mathcal{G}_d$ and the gliding part $\mathcal{G}_g$ depending on whether $H_p^2 x^n > 0 $ or $H_p^2 x^n < 0 $. 

Next we are ready to define the generalized broken bicharacteristic of $\sq$, 
see \cite{Melrose1978, Melrose1982,MR2304165} and \cite[Definition 1.1]{Vasy2008}.
\begin{df}[{\cite[Definition 24.3.7]{MR2304165}}]\label{def_gb}
    Let $I \subset \mathbb{R}$ be an open interval and $B \subset I$ is a discrete subset. 
    A generalized broken bicharacteristic arc of $\sq$ is a map $\nu: I \rightarrow \pi_b^{-1}(\mathcal{G} \cup \mathcal{H})$ satisfying the following properties:
    \begin{enumerate}[(1)]
        \item $\nu(t)$ is differentiable and $\nu'(t) = H_p(\nu(t))$, if $\nu(t) \in T^*\intM$ 
        \item $\nu(t)$ is differentiable and $\nu'(t) = H_p^G(\nu(t))$ , if $\nu(t) \in \pi_b^{-1}( \mathcal{G}^2 \setminus \mathcal{G}_d)$, see \cite[Definition 24.3.6]{MR2304165} for the vector field $H_p^G(\nu(t))$,
        \item every $t \in B$ is isolated, and $\nu(s) \in  T^*\intM$ if $|s-t|$ is small enough and $s \neq t$.
        The limits $\nu(t\pm 0)$ exist and are different points in the same fiber of $\partial T^*M$. 
    \end{enumerate}
The continuous curve $\dot{\nu}$ obtained by mapping $\nu$ into $\compchar$ by $\pi_b$  is called a generalized broken bicharacteristic. 
    \end{df}  

If $\nu$ is a generalized bicharacteristic arc contained in $\gsetint \cup \mathcal{H}$, then we call it a broken bicharacteristic arc, see \cite[Definiton 24.2.2]{MR2304165}, which is roughly speaking the union of bicharacteristics of $\sq$ over $\intM$ with reflection points on the boundary. 
By the definition, a broken bicharacteristic arc arrives transversally to $\partial M$ at $\nu(t-0)$ and then leaves the boundary transversally from the reflected point $\nu(t+0)$, with the same projection in $T^*\partial M \setminus 0$. 
The image of such $\nu$ under $\pi_b$ is called a broken bicharacteristic.   

In this paper, with the assumption that $\partial M$ is null-convex, 
the generalized broken bicharacteristics we consider in Section \ref{Sec_threewaves} and \ref{Sec_fourwaves} are contained in $\gsetint \cup \mathcal{H}$, and therefore are all broken bicharacteristics. 
We denote
the broken bicharacteristic arc of $\sq$ that contains the covector $(y, \eta) \in L^*M$ by $\Theta^b_{y, \eta}$. 
According to \cite[Corollary 24.3.10]{MR2304165}, the arc $\Theta^b_{y, \eta}$ is unique for each 
$\pi_b(y, \eta) \notin \mathcal{G}^\infty$. 
The following lemma shows that with proper assumptions on the boundary, a generalized bicharacteristics passing a point in $\gsetint \cup \mathcal{H}$ is always contained in $\gsetint \cup \mathcal{H}$, i.e., is always a broken bicharacteristic. 

\begin{lm}\label{lm_gh}
    Let $(M,g)$ be a Lorentzian manifold with timelike 
    boundary $\partial M$. 
    Suppose $\partial M$ is null-convex, see (\ref{def_nconvex}). 
    If $\dot{\nu}$ is a future pointing generalized bicharacteristic with $\dot{\nu}(0) \in \gsetint \cup \mathcal{H}$, then $\dot{\nu} \subset \gsetint \cup \mathcal{H}$ and therefore it is a future pointing broken characteristic. 
\end{lm}

\begin{proof}
    If $\dot{\nu}(0) \in \mathcal{H}$, then there are two different points in $\pi_b^{-1}(\dot{\nu})(0) \cap \Char({\sq})$. 
    In this case suppose $\nu^\pm \in T^\pm_{\partial M}  M$ such that $\pi_b(\nu^\pm) = \dot{\nu}(0)$.
    By Definition \ref{def_gb}, the generalized bicharacteristic arc  $\nu $ of $\dot{\nu}$ has $\nu(0-) = \nu^+$ and $\nu(0+) = \nu^-$. 
    It leaves $\partial M$ transversally after $t >0$ and one has $\dot{\nu}(t) \in \gsetint$ for small enough positive $t$. 
    
    Now we consider the second case where $\dot{\nu}(0) \in \gsetint$. 
    The generalized bicharacteristic arc  $\nu $ of $\dot{\nu}$ has $\nu(0) \in T^* \intM$ and therefore near $\nu(0)$ it is the null bicharacteristic of $\sq$. 
    Let $\pi: T^*M \rightarrow M$ be the natural projection. 
    Then $\pi(\nu)(t)$ near $t = 0$ is a light-like geodesic in $\intM$. 
    If $\pi(\nu)(t)$ hits the boundary at $\pi(\nu)(t_0) = p_0$, 
    then by \cite[Proposition 2.4]{Hintz2017} the null-convex boundary implies that the intersection is transversal with the tangent vector $v^+$ such that $(p_0, v^+) \in T^+_{\partial M} M$. 
    Let $v^- = v^+  - 2 g(v^+, n) n$ and it follows that
    $
    g(v^-, n) < 0, \  g(v^-, v^-) = 0
    $
    and $\pi_b(v^+) = \pi_b(v^-)$.
    Thus, $\dot{\nu}(t_0) \in \mathcal{H}$ and we arrive at the first case. 
    
    Combining the two cases, we have $\dot{\nu}$ is always contained in $\gsetint \cup \mathcal{H}$. 
    \end{proof}
Let $\dot{\mathcal{D}}'(M)$ be the set of distributions supported in $M$. 
The boundary wave front set $\wfset_b(u)$ of a distribution $u \in \dot{\mathcal{D}'}(M)$ is defined by 
\[
\wfset_b(u)  = \bigcap \pi_b(\Char (B)),
\]
where the intersection takes for any properly supported $B \in \Psi^0_b(M)$ and $Bu \in \dot{\mathcal{A}}(M)$.  
Here $\Psi^0_b(M)$ is the set of b-pseudodifferential operators on $M$ of order zero, for more details see \cite[Definition 18.3.18]{MR2304165}. 
The set $\dot{\mathcal{A}}(M) \equiv \bigcup_m I^m(M, \partial M)$ contains distributions that are smooth in $\intM$ and have tangential smoothness at $\partial M$, where $I^m(M, \partial M)$ denotes the class of conormal distributions on $\partial M$ of order $m$. 
Note that away from the boundary it coincides with the usual definition of wave from set, i.e. $\wfset_b(u)|_{\intM} = \wfset(u|_{\intM})$.  
We present the result about the singularities of solutions to the boundary value problem (\ref{bvp_qg}) in the following proposition according to \cite{Melrose1978, Melrose1982}, \cite[Theorem 8.1]{Vasy2008}, and  \cite{MR2304165}. 

\begin{pp} \label{pp_wfb}
    Let $h \in \dot{\mathcal{D}'}(M)$ and $w = \Qb_g(h)$ be the solution to the boundary value problem (\ref{bvp_qg}) in $M$. 
    Then \[
    \wfset_b(w) \setminus \wfset_b(h) \subset \compchar
    \] is a union of maximally extended generalized characteristics of $\sq$. 
    
 \end{pp}
Combining this proposition and Lemma \ref{lm_gh}, we have the following corollary. 
\begin{corollary}
In particular, if $\partial M$ is timelike and null-convex 
with 
$\wfset_b(h)$ contained in $\gsetint \cup \mathcal{H}$, then 
\[
\wfset_b(w) \subset \gsetint \cup \mathcal{H}
\]
and is a union of broken bicharacteristics. 
\end{corollary}

\subsection{The asymptotic expansion}\label{subsec_assyp}
Let $f = \sum_{j = 1}^J \epsilon_j f_j$, where $J = 3$ or $4$. 
The small boundary data $f_j$ are properly chosen as before.
Let $v_j$ solve the boundary value problem (\ref{eq_v1})
and recall we write $w  = \Qbg(h)$ if $w$ solves the boundary value problem (\ref{bvp_qg}).

Let $v = \sum_{j=1}^J \ep_j v_j$ and by (\ref{eq_problem}) we have
\[
\sq (p -v) = F(x, p,\partial_t p, \partial^2_t p). 
\]
It follows from (\ref{eq_nlterm}) that
\begin{align}\label{expand_u}
p &= v + \sum_{m=1}{\Qbg( \beta_{m+1}(x) \partial_t^2 (p^{m+1}))}, \nonumber \\
& = v + {A_2 + A_3 + A_4 + \dots},
\end{align}
where we write the term $\sum_{m=1}\Qbg(\beta_{m+1}(x) \partial_t^2 (p^{m+1}))$ by the order of $\epsilon$-terms, such that $A_2$ denotes the terms with $\epsilon_i \epsilon_j$, 
$A_3$ denotes the terms with $\epsilon_i \epsilon_j \epsilon_k$, and $A_4$ denotes the terms with $\epsilon_i \epsilon_j \epsilon_k \epsilon_l$, for $1 \leq i,j,k,l \leq J$. 
By (\ref{expand_u}), one can find the expansions of $A_2, A_3, A_4$ as
\begin{align*}
A_2 &= \Qbg(\beta_2 \partial_t^2(v^2)),\\
A_3 &= \Qbg(2\beta_2 \partial_t^2(vA_2)+\beta_3 \partial_t^2(v^3))\\
A_4 &= \Qbg(2 \beta_2 \partial_t^2(vA_3) + \beta_2\partial_t^2(A_2A_2) + 3\beta_3\partial_t^2(v^2 A_2) + \beta_4\partial_t^2(v^4)).
\end{align*}
For $N \geq 5$, we can write
\begin{align*}
A_N= \Qbg(\beta_N\partial_t^2(v^N)) + \mathcal{Q}_{N}(\beta_2, \beta_3, \ldots, \beta_{N-1}), 
\end{align*}
where $\mathcal{Q}_{N}(\beta_2, \beta_3, \ldots, \beta_{N-1})$ contains the terms involved only with $\beta_2, \ldots, \beta_{N-1}$. 
Note that $v$ appears $j$ times in each $A_j$, $j = 2, 3, 4$. 
Therefore, we introduce the notation $A_2^{ij}$ to denote the result if we replace $v$ by $v_i, v_j$ in $A_2$ in order, and similarly the notations
$A_3^{ijk}$, $A_4^{ijkl}$,
such that 
\[
A_2 = \sum_{i,j} \ep_i \ep_j A_2^{ij}, \quad 
A_3 = \sum_{i,j, k} \ep_i\ep_j\ep_k  A_3^{ijk},\quad  
A_4 =\sum_{i,j, k,l} \ep_i\ep_j\ep_k \ep_l A_4^{ijkl}.
\] 
More explicitly, we have
\begin{align}\label{eq_A}
\begin{split}
A_2^{ij} &= \Qbg(\beta_2 \partial_t^2(v_iv_j)),\\
A_3^{ijk} &= \Qbg(2\beta_2 \partial_t^2(v_iA_2^{jk})+\beta_3 \partial_t^2(v_i v_j v_k))\\
A_4^{ijkl} &= \Qbg(2 \beta_2 \partial_t^2(v_iA_3^{jkl}) + \beta_2\partial_t^2(A_2^{ij}A_2^{kl}) + 3\beta_3\partial_t^2(v_i v_j A_2^{kl}) + \beta_4\partial_t^2(v_iv_jv_kv_l)).
\end{split}
\end{align}

\section{The nonlinear interaction of three waves}\label{Sec_threewaves}
In this section, we consider the interaction of three distorted plane waves. 
Suppose $(x_j, \xi_j)_{j=1}^3$ intersect regularly at $q \in \intM$ and are casually independent as in (\ref{assump_xj}).
Let $K_j = K(x_j, \xi_j, s_0)$ and $\Lambda_j = \Lambda(x_j, \xi_j, s_0)$ be defined as in Section \ref{subsec_planewaves}. 
With sufficiently small $s_0 >0$, we can assume the submanifolds $K_1, K_2, K_3$ intersect 3-transversally, i.e.,
\begin{itemize}
    \item[(1)] $K_i$ and $K_j$ intersect transversally at a codimension $2$ submanifold $K_{ij}$, for $1 \leq i < j \leq 3$; 
    \item[(2)] $K_1, K_2, K_3$ intersect transversally at a codimension $3$ submanifold $K_{123}$.  
\end{itemize}
By Lemma \ref{lm_intM}, in $\nxxi \cap \ntxxi$, one has $K_{ij}, K_{ijk} \subset \intM$ with sufficiently small $s_0$, and therefore
$
\pi_b(\Lambda_{ij}), \pi_b(\Lambda_{ijk}) \subset \gsetint
$
for $ 1 \leq i < j \leq 3$. 

Recall we can construct distorted waves $u_j$ associated with $(x_j, \xi_j)_{j=1}^3$ such that
\[
u_j \in I^{\mu}(\Lambda(x_j, \xi_j, s_0)), \quad j = 1,2,3,
\] 
solves the linearized wave problem in $M$, i.e.,  $\sq u_j \in C^\infty(M)$, with the principal symbol nonvanishing along $\gamma_{x_j, \xi_j}(\mathbb{R}_+)$. 
Let $f_j = u_j|_{\partial M}$ and $f = \sum_{j = 1}^3 \ep_j f_j$  as the Dirichlet data for (\ref{eq_problem}).
Suppose $v_j$ solves (\ref{eq_v1}).
It follows that $v_j$ is equal to $u_j$ module $C^\infty(M)$. 
We define
\[
\mathcal{U}^{(3)} = \partial_{\epsilon_1}\partial_{\epsilon_2}\partial_{\epsilon_3} u |_{\epsilon_1 = \epsilon_2 = \epsilon_3=0},
\]
and combine (\ref{expand_u}), (\ref{eq_A}) to have
\begin{align*}
\mathcal{U}^{(3)} 
&=  \sum_{(i,j,k) \in \Sigma(3)} A_3^{ijk} 
=   \sum_{(i,j,k) \in \Sigma(3)} \Qbg(2\beta_2 \partial_t^2(v_iA_2^{jk})+\beta_3 \partial_t^2(v_i v_j v_k)).
\end{align*}
Note that $\mathcal{U}^{(3)}$ is not the third order linearization of $\Lambda_F$ but they are  related by 
\begin{align}\label{eq_LambdaU3}
\partial_{\epsilon_1}\partial_{\epsilon_2}\partial_{\epsilon_3} \Lambda_F (f) |_{\epsilon_1 = \epsilon_2 = \epsilon_3=0} = \lge \nu, \nabla \mathcal{U}^{(3)} \rge|_{\partial M}.
\end{align}
As in \cite{Hintz2020}, we introduce the trace operator $\mathcal{R}$ on $\pM$.
It is an FIO and maps distributions in $\mathcal{E}'(M)$ whose singularities are away from $N^*(\pM)$ to $\mathcal{E}'(\pM)$, see \cite[Section 5.1]{Duistermaat2010}.
Notice for any timelike covector $(y_|, \eta_|) \in T^* \partial M \setminus 0$, there is exactly one outward pointing lightlike covector $(y, \eta^+)$ and one inward pointing lightlike covector $(y, \eta^-)$ satisfying $y_| = y,\  \eta_| = \eta^\pm|_{T^*_{y} \partial M}$. 
The trace operator $\mathcal{R}$ has
a nonzero principal symbol at such $(y_|, \eta_|, y, \eta^+)$ or $(y_|, \eta_|, y, \eta^-)$ .  

We emphasize that $v_j \in \dot{\mathcal{D}}'(M)$ and  
we can always identify it as en element of $\mathcal{D}'(\tM)$.
Particularly in $\nxxi \cap \ntxxi$, it can be identified as a conormal distribution in $\tM$. 
Note that $v_iv_j$ or $v_iv_jv_k$ has singularities away from $N^*\partial M$, since we have $K_{ij}, K_{ijk} \subset \intM$.
Then we use Proposition \ref{pp_wfb} and its corollary to analyze the singularities of $\Qbg(v_iv_j)$ or $\Qbg(v_iv_jv_k)$. 
This is the same argument we use in \cite[Section 4, 5]{Uhlmann2021a}. 
We present the proof in the following for completeness. 
\subsection{The analysis of $A_2^{ij}$} 
First we analyze the singularities of 
 \[
A_2^{ij} =\Qbg(\beta_2 \partial_t^2(v_iv_j)), \quad 1 \leq i < j \leq 3.
\] 
By \cite[Lemma 3.3]{Lassas2018} and \cite[Lemma 4.1]{Wang2019}, one has 
\[
\beta_2 \partial_t^2(v_iv_j) \in I^{\mu+2, \mu+1}(\Lambda_{ij}, \Lambda_i) + I^{\mu+2, \mu+1}(\Lambda_{ij}, \Lambda_j).
\]
In the following, let $(q, \zj) \in \Lambda_j$ and we write $\zeta^{(j)} = (\zjz, \zeta_1^{(j)}, \zeta_2^{(j)}, \zeta_3^{(j)})$, for $j=1,2,3$.   
With $K_i, K_j$ intersecting transversally, any $(q, \zeta) \in \Lambda_{ij}$ has a unique decomposition $\zeta = \zi + \zj$. 
Away from 
$\Lambda_i$ and $\Lambda_j$, the principal symbol of $\beta_2 \partial_t^2(v_iv_j)$ equals to
\[
-(2\pi)^{-1}\beta_2 (\zeta_0^{(i)}+ \zeta_0^{(j)})^2 {\sigmp}(v_i) (q, \zeta^{(i)}) {\sigmp}(v_j) (q, \zeta^{(j)})
\] 
at $(q, \zeta) \in \Lambda_{ij}$.

Note that $\beta_2 \partial_t^2(v_iv_j)$ is also a distribution supported in $M$.
Its boundary wave front set is contained in $\pi_b(\Lambda_i \cup \Lambda_j\cup {\Lambda_{ij}})$ and thus 
as a subset of $\mathcal{G}^{\text{int}} \cup \mathcal{H}$.
Then by Proposition \ref{pp_wfb} and its corollary, the set $\wfset_b(A_2^{ij})$ is contained in the union of $\pi_b(\Lambda_i \cup \Lambda_j\cup {\Lambda_{ij}})$ and their flow out under the broken bicharacteristics.
We notice that away from $J^+(\gamma_{x_i, \xi_i}(t_i^b)) $ and $ J^+(\gamma_{x_i, \xi_i}(t_j^b)) $, there are no new singularities produced by the flow out. 

In $\ntxxi$, we identify $\Qbg$ by $Q_g$, the causal inverse of $\sq$ on $\tM$, to have
\[
A_2^{ij} \in I^{\mu+1, \mu}(\Lambda_{ij}, \Lambda_i) + I^{\mu+1, \mu}(\Lambda_{ij}, \Lambda_j),
\]
by \cite[Lemma 3.4]{Lassas2018}. 
Here we regard the restriction of $A_2^{ij}$ to $\nxxi \cap \ntxxi$ as a distribution in  $ \tM$.
Additionally, 
at $(q, \zeta) \in \Lambda_{ij}$ away from $\Lambda_i$ and $\Lambda_j$,  
the principal symbol equals to
\begin{align*}
{\sigmp}(A_2^{ij}) (q, \zeta) &= - (2\pi)^{-1}  \beta_2 \frac{(\zeta_0^{(i)}+ \zeta_0^{(j)})^2}{|\zeta^{(i)} + \zeta^{(j)}|_{g^*}^{2}}
{\sigmp}(v_i) (q, \zeta^{(i)}) {\sigmp}(v_j) (q, \zeta^{(j)}).
\end{align*}

\subsection{The analysis of $A_3^{ijk}$}
Recall
\[
A_3^{ijk} = \Qbg(2\beta_2 \partial_t^2(v_iA_2^{jk})+\beta_3 \partial_t^2(v_i v_j v_k)). 
\]
The following proposition describes the singularities of $A_3^{ijk}$ produced by three waves interaction. 

\begin{pp}\label{pp_Aijk}
    Suppose $K_i, K_j, K_k$ intersect 3-transversally at $K_{ijk}$. 
    Then in $\nxxi \ \cap \ \ntxxi$, for the definition see (\ref{def_nxxi}) and (\ref{def_ntxxi}), we have the following statements. 
    \begin{enumerate}[(a)]
        \item  There is a decomposition $A_3^{ijk} = \tw_0 + \tw_1 + \tw_2 + \tw_3$ with 
        \begin{align}\label{eq_Aijk}
        \begin{split}
        &\tw_0 \in I^{3\mu + 3}(\Lambda_{ijk}), 
        \quad \wfset_b(\tw_1) \subset \pi_b((\Lambda_{ijk}^{g}(\epsilon) \cap \Lambda_{ijk}) \cup \Lambda_{ijk}^{b}),\\
        &\wfset_b(\tw_2) \subset \pi_b(\Lambda^{(1)} \cup (\Lambda^{(1)}(\epsilon) \cap \Lambda_{ijk}) \cup \Lambda_{ijk}^{b}),
        \quad  \wfset(\tw_3)\subset \Lambda^{(1)} \cup \Lambda^{(2)}.
        \end{split}
        \end{align}
        In particular, for $(q, \zeta) \in \Lambda_{ijk}$, the leading term $\tw_0$ has the principal symbol (\ref{eq_tw0ps}).
        \item Let $(y, \eta) \in \LcMpo$ be a covector
        lying along the forward null-bicharacteristic starting at 
        $(q, \zeta) \in \Lambda_{ijk}$. 
        Suppose $(y, \eta)$ is away from $\Lambda^{(1)}$. 
        Then $\sigmp(\mathcal{U}^{(3)})(y, \eta)$
        is given in (\ref{eq_lambdaps}).
        \item 
        Let $(y_|, \eta_|)$ be the projection of $(y, \eta)$ on the boundary. 
        Moreover, we have 
        \[
        {\sigmp}({\partial_{\epsilon_1}\partial_{\epsilon_2}\partial_{\epsilon_3} \Lambda_F(f) |_{\epsilon_1 = \epsilon_2 = \epsilon_3=0}})(y_|, \eta_|) = \iota  \lge \nu, \eta \rge_g
        {\sigmp}(\mathcal{U}^{(3)})(y_|, \eta_|). 
        \]
    \end{enumerate}
\end{pp}
\begin{proof}
We write $A_3^{ijk}  = B_3^{ijk} + C_3^{ijk}$, where
\[
B_3^{ijk} = \Qbg(2\beta_2 \partial_t^2(v_iA_2^{jk})), \quad 
C_3^{ijk} = \Qbg(\beta_3 \partial_t^2(v_i v_j v_k))). 
\]
In the following, let $(q, \zeta^{(m)}) \in \Lambda_m, m=i,j,k$ and $(q, \zeta) \in \Lambda_{ijk}$ with $\zeta = \zi + \zj + \zk$. 
The transversal intersection implies the decomposition of $\zeta$ is unique. 

For $B_3^{ijk}$, since we are away from $J^+(\gamma_{x_i, \xi_i}(t_i^b)) $ and $ J^+(\gamma_{x_j, \xi_j}(t_j^b))$,
first one can write
$v_iA_2^{jk} = w_0 + w_1 + w_2$
 using \cite[Lemma 3.6]{Lassas2018}, where
\begin{align}\label{eq_Bijk}
\begin{split}
\ & w_0 \in I^{3\mu + 2}(\Lambda_{ijk}), \quad \wfset(w_2) \subset \Lambda^{(1)} \cup \Lambda^{(2)},\\
\ &  \wfset(w_1) \subset \Lambda^{(1)} \cup (\Lambda^{(1)}(\epsilon) \cap \Lambda_{ijk}).
\end{split}
\end{align}
The leading term $w_0$ has principal symbol 
\begin{align*}
{\sigmp}(w_0) (q, \zeta) = 
-(2\pi)^{-2}
\beta_2
\frac{(\zjz + \zkz)^2}{|\zeta^{(j)} + \zeta^{(k)}|^2_{g^*}} 
\prod_{m= i,j,k}
{\sigmp}(v_m) (q, \zeta^{(m)}).
\end{align*}

Next, we apply $2\beta_2 \partial_t^2$ to $w_0, w_1, w_2$ respectively. The wave front sets remain the same and we have $2\beta_2 \partial_t^2 w_0 \in I^{3\mu + 4}(\Lambda_{ijk})$ with
\begin{align}\label{ps_bpw0}
{\sigmp}(2\beta_2 \partial_t^2 w_0) (q, \zeta) = 
(2\pi)^{-2}(\zeta_0)^2 ( 2\beta^2_2 
\frac{(\zjz + \zkz)^2}{|\zeta^{(j)} + \zeta^{(k)}|^2_{g^*}} )
\prod_{m= i,j,k}
{\sigmp}(v_m) (q, \zeta^{(m)}).
\end{align}
Then, we apply $\Qbg$.  
By Proposition \ref{pp_wfb} and its corollary, the set $\wfset_b(\Qbg(w_k))$ is a union of  $\wfset_b(w_k)$ and its flow out under the broken bicharacteristics.  
For this first term, it follows that
\[
\wfset_b(\Qbg (2\beta_2 \partial_t^2w_0) )\subset\pi_b(\Lambda_{ijk} \cup \Lambda_{ijk}^{b}). 
\]
To find its principal symbol,
we decompose it as
$ \Qbg (2\beta_2 \partial_t^2w_0)  = u^{\text{inc}} + u^{\text{ref}}$,
where 
$u^{\text{inc}} $ is the incident wave before the reflection on the boundary and 
$u^{\text{ref}}$ is the reflected one. 
Recall $\tQ_g$ is the causal inverse of $\sq$ on $\tM$.
Then $ u^{\text{inc}}  = \tQ_g(2\beta_2 \partial_t^2w_0)$ and by \cite[Proposition 2.1]{Greenleaf1993} we have
\begin{align}\label{r_qw0}
u^{\text{inc}} =\tQ_g(2\beta_2 \partial_t^2w_0) \in I^{3 \mu + \frac{5}{2}, -\frac{1}{2}}  ( \Lambda_{ijk} ,\Lambda_{ijk}^{g} ).
\end{align}
The principal symbol at $(q, \zeta) \in \Lambda_{ijk}$ equals to
\begin{align}\label{ps_qw0}
{\sigmp}(u^{\text{inc}})(q, \zeta) = 
|\zeta^{(i)} + \zeta^{(j)} + \zeta^{(k)}|^{-2}_{g^*}
{\sigmp}(2\beta_2 \partial_t^2w_0) (q, \zeta).
\end{align}
If $(y, \eta) \in \LcMpo$ lies along the forward null-bicharacteristic starting from $(q, \zeta) \in \Lambda_{ijk}$, then $(y,\eta)$ belongs to $\Lambda_{ijk}^b$ and is the first point there to hit the boundary.
The principal symbol of $u^{\text{inc}} $ at $(y, \eta)$ is
\begin{align}\label{ps_qw0_b}
{\sigmp}(u^{\text{inc}})(y, \eta)  =  {\sigmp}(\tQ_g)(y, \eta, q, \zeta)   
{\sigmp}(2\beta_2 \partial_t^2w_0)(q, \zeta).
\end{align}
Near $(y, \eta)$, the reflected wave $u^{\text{ref}}$ satisfies
\[
\sq u^{\text{ref}}= 0, \quad   (u^{\text{inc}} +  u^{\text{ref}})|_{\partial M} = 0,
\]
where the second equation comes from the Dirichlet boundary condition.
In a small conic neighborhood of $(y, \eta)$, we write
\[
u^{\bullet} = (2 \pi)^{-3} \int e^{i \phi^\bullet(x, \theta)} a^\bullet (x, \theta) \diff \theta,
\]
for $\bullet = \text{inc, ref}$, with suitable amplitude $a^\bullet (x, \theta)$ satisfying the transport equation and the phase functions $\phi^\bullet(x, \theta)$ satisfying the eikonal equation.
Note that
$ \partial_\nu \phi^\text{ref}|_{\pM} = - \partial_\nu \phi^\text{inc}|_{\pM}
$
near $(y, \eta)$,
since the incident wave is the incoming solution and the reflected wave is the outgoing one.
The Dirichlet boundary condition implies $a^{\text{inc}}|_{\pM} = -a^{\text{ref}}|_{\pM}$ near $(y, \eta)$.
In local coordinates, if we omit the half-density factor, then we have
\[
\sigmp(\lge \nu, \nabla u^{\bullet} \rge|_{\partial M})(\yb, \etab) = \iota(\partial_\nu \phi^{\bullet}) a^{\bullet}(y_|, \eta_|),
\]
where $\bullet  = \text{inc, ref}$ and $(y_|, \eta_|)$ is the projection of the lightlike covector $(y, \eta)$ on the boundary.
It follows that
\[
\sigmp(\lge \nu, \nabla u^{\text{inc}} \rge|_{\partial M})(\yb, \etab)= \sigmp(\lge \nu, \nabla u^{\text{ref}} \rge|_{\partial M})(\yb, \etab),
\]
and therefore
\begin{align}\label{eq_rqw0}
\sigmp(\Qbg(2\beta_2 \partial_t^2w_0))(y,\eta)
& =  2 \sigmp( u^{\text{inc}})(y, \eta),\\
\sigmp(\lge \nu, \nabla( \Qbg(2\beta_2 \partial_t^2w_0)) \rge|_{\partial M})(\yb, \etab)
&= 2\sigmp(\lge \nu, \nabla u^{\text{inc}} \rge|_{\partial M})(\yb, \etab) \nonumber  \\
& = 2 \iota  \sigmp (\mathcal{R})(y_|, \eta_|, y, \eta) \lge \nu, \eta \rge_g \sigmp( u^{\text{inc}})(y, \eta),\nonumber
\end{align}
where the last equality is by \cite[Lemma 4.1]{Wang2019}. 

For $\Qbg(2\beta_2 \partial_t^2 w_1)$, 
first we identify $w_1$ as an element in $\dot{D}'(M)$.  
Its boundary wave front set in $M$ is a subset of $\pi_b(\Lambda^{(1)}) \cup \pi_b(\Lambda^{(1)}(\epsilon) \cap \Lambda_{ijk})$. 
Then $\wfset_b(\Qbg(2\beta_2 \partial_t^2 w_1)) $ is contained in the union of this set and its flow out under broken bicharacteristics, i.e. $\pi_b(\Lambda^{(1)} \cup (\Lambda^{(1)}(\epsilon) \cap \Lambda_{ijk}) \cup \Lambda_{ijk}^{b})$,
away from $\bigcup_{j=1}^3 J^+(\gamma_{x_j, \xi_j}(t_j^b))$. 
When $\ep$ goes to zero, the boundary wave front set tends to $\pi_b(\Lambda^{(1)} \cup \Lambda_{ijk}^{b})$.

For $\Qbg(2\beta_2 \partial_t^2 w_2)$, the boundary wave front set of $w_2$ in $M$ is a subset of $\pi_b(\Lambda^{(1)} \cup \Lambda^{(2)})$. 
Then $\wfset_b(\Qbg(2\beta_2 \partial_t^2w_2)) $ is contained in the union of this set and its flow out under the broken bicharacteristics, i.e., 
$\pi_b(\Lambda^{(1)} \cup \Lambda^{(2)})$, away from $\bigcup_{j=1}^3 J^+(\gamma_{x_j, \xi_j}(t_j^b))$.

Thus, we write $B_3^{ijk} =  \tw_0 + \tw_1 + \tw_2 + \tw_3$ with 
\begin{align*}
\begin{split}
&\tw_0 \in I^{3\mu + 2}(\Lambda_{ijk}), 
\quad   \wfset_b(\tw_1) \subset \pi_b((\Lambda_{ijk}^{g}(\epsilon) \cap \Lambda_{ijk}) \cup \Lambda_{ijk}^{b}),\\
&\wfset_b(\tw_2) \subset \pi_b(\Lambda^{(1)} \cup (\Lambda^{(1)}(\epsilon) \cap \Lambda_{ijk}) \cup \Lambda_{ijk}^{b}),
\quad  \wfset(\tw_3)\subset \Lambda^{(1)} \cup \Lambda^{(2)},
\end{split}
\end{align*}
where the first two terms $\tw_0, \tw_1$ come from applying \cite[Lemma 3.9]{Lassas2018} to (\ref{r_qw0}).
Indeed, \cite[Lemma 3.9]{Lassas2018} implies that we can decompose (\ref{r_qw0}) as a conormal distribution supported in $\Lambda_{ijk}$ and a distribution microlocally supported in $\Lambda_{ijk}^g(\epsilon)$.
Since $u^{\text{inc}}$ in (\ref{r_qw0}) has wave front set contained in $\Lambda_{ijk} \cup \Lambda_{ijk}^g$, the second distribution in the decomposition must be microlocally supported in $(\Lambda_{ijk}^{g}(\epsilon) \cap \Lambda_{ijk}) \cup \Lambda_{ijk}^{g}$. 
Then we include the wave front set after $\Lambda_{ijk}^{g}$ hits the boundary for the first time to have $\wfset_b(\tw_1)$ in (\ref{eq_Aijk}).

For $C_3^{ijk}$, note that it has the same pattern as $B_3^{ijk}$. 
In this case, the leading term in (\ref{eq_Bijk}) is a conormal distribution in $I^{3\mu + 2}(\Lambda_{ijk})$, with principal symbol 
\begin{align}\label{eq_Cijk}
(2\pi)^{-2}
\prod_{m= i,j,k}
{\sigmp}(v_m) (q, \zeta^{(m)})
\end{align}
at $(q, \zeta) \in \Lambda_{ijk}$. 
We apply $\beta_3 \partial_t^2$ to have 
$\beta_3 \partial_t^2 w_0 \in I^{3\mu + 4}(\Lambda_{ijk})$ with
\begin{align}\label{ps_bpw0_C}
{\sigmp}(\beta_3 \partial_t^2 w_0) (q, \zeta) = 
-(2\pi)^{-2} (\zeta_0)^2 \beta_3 
\prod_{m= i,j,k}
{\sigmp}(v_m) (q, \zeta^{(m)}).
\end{align}
Then we apply $\Qbg$ and follow the same analysis as above. 

Combining the analysis of $B_3^{ijk}$ and $C_3^{ijk}$, we conclude that we can write 
$A_3^{ijk} = \tw_0 + \tw_1 + \tw_2 + \tw_3$ with $\tw_j$ satisfying (\ref{eq_Aijk}) for $j = 0,1,2,3$. 
In particular, 
the leading term $\tw_0$ has principal symbol at $(q, \zeta) \in \Lambda_{ijk}$ given by 
\begin{align}\label{eq_tw0ps}
&{\sigmp}(\tw_0)(q, \zeta) \\
= & (2\pi)^{-2}
\frac{(\zeta_0)^2}{|\zeta^{(i)} + \zeta^{(j)} + \zeta^{(k)}|^{2}_{g^*}}
(2\beta^2_2 
\frac{(\zjz + \zkz)^2}{|\zeta^{(j)} + \zeta^{(k)}|^2_{g^*}} - \beta_3)
\prod_{m= i,j,k}{\sigmp}(v_m) (q, \zeta^{(m)}).  \nonumber
\end{align}

Suppose $(y, \eta) \in \LcMpo$
lies along the forward null-bicharacteristic starting from $(q, \zeta) \in \Lambda_{ijk}$ and is 
away from $\Lambda^{(1)} \cup \Lambda^{(2)} \cup \Lambda^{(3)}$.
We combine (\ref{ps_bpw0}), 
(\ref{ps_qw0_b}), (\ref{eq_rqw0}), and (\ref{ps_bpw0_C}) to have 
\begin{align} \label{eq_lambdaps}
&{\sigmp}(\mathcal{U}^{(3)})(y, \eta)
= 2  (2\pi)^{-2}  {\sigmp}(\tQ_g)(y, \eta, q, \zeta) 
(\zeta_0)^2 \mathcal{Q}(\zeta^{(1)}, \zeta^{(2)}, \zeta^{(3)}) 
 \prod_{m=1}^3{\sigmp}(v_m) (q, \zeta^{(m)}), 
\end{align}
where 
\begin{align*}
\mathcal{Q}(\zeta^{(1)}, \zeta^{(2)}, \zeta^{(3)}) 
&= \sum_{(i,j,k) \in \Sigma(3)}
2\frac{(\zjz + \zkz)^2}{|\zeta^{(j)} + \zeta^{(k)}|^2_{g^*}} \beta^2_2  - \beta_3. 
\end{align*}
\end{proof}

\section{The nonlinear interaction of four waves}\label{Sec_fourwaves}
In this section, we consider the interaction of four distorted plane waves. 
Suppose $(x_j, \xi_j)_{j=1}^4$ intersect regularly at $q$ and are casually independent as in (\ref{assump_xj}).
Let $K_j = K(x_j, \xi_j, s_0)$ and $\Lambda_j = \Lambda(x_j, \xi_j, s_0)$ be defined as in Section \ref{subsec_planewaves}. 
With sufficiently small $s_0 >0$, we can assume $K_1, K_2, K_3, K_4$ intersect 4-transversally at $q$, see Definition \ref{df_intersect}.
By Lemma \ref{lm_intM}, in $\nxxi \cap \ntxxi$ we must have $K_{ij}, K_{ijk}$ and $q$ contained in $\intM$,
if $s_0>0$ is small enough.
We consider distorted waves $u_j$ associated with $(x_j,\xi_j)_{j=1}^4$ such that
 \[
u_j \in I^{\mu}(\Lambda(x_j, \xi_j, s_0)), \quad j = 1,2,3,4
\] 
solves the linearized wave problem in $M$, i.e.,  $\sq u_j \in C^\infty(M)$, with the principal symbol nonvanishing along $\gamma_{x_j, \xi_j}(\mathbb{R}_+)$.   
Let $f_j = u_j|_{\partial M}$ and we consider the Dirichlet data $f = \sum_{j = 1}^4 \ep_j f_j$ for the semilinear boundary value problem (\ref{eq_problem}).
Suppose $v_j$ solves (\ref{eq_v1}).
It follows that $v_j$ is equal to $u_j$ module $C^\infty(M)$. 
We define
\begin{align*}
\ufour = \partial_{\epsilon_1}\partial_{\epsilon_2}\partial_{\epsilon_3}\partial_{\epsilon_4} u |_{\epsilon_1 = \epsilon_2 = \epsilon_3 = \epsilon_4=0},
\end{align*}
and combine (\ref{expand_u}), (\ref{eq_A}) to have 
\begin{align}\label{u4}
\ufour = \sum_{(i,j,k,l) \in \Sigma(4)} \mathcal{A}_4^{ijkl} 
=\sum_{(i,j,k,l) \in \Sigma(4)} &
 \Qbg(2 \beta_2 \partial_t^2(v_iA_3^{jkl}) + \beta_2\partial_t^2(A_2^{ij}A_2^{kl}) +\\ 
 &\quad \quad \quad + 3\beta_3\partial_t^2(v_i v_j A_2^{kl}) + \beta_4\partial_t^2(v_iv_jv_kv_l)). \nonumber
\end{align}
Note that $\mathcal{U}^{(4)}$ is not the fourth order linearization of $\Lambda_{F}$ but they are related by 
\begin{align}\label{eq_LambdaU4}
\partial_{\epsilon_1}\partial_{\epsilon_2}\partial_{\epsilon_3}{\epsilon_4} \Lambda_F (f) |_{\epsilon_1 = \epsilon_2 = \epsilon_3=\epsilon_4=0} = \lge \nu, \nabla \mathcal{U}^{(4)} \rge|_{\partial M}.
\end{align}
The following proposition describes the singularities of $\ufour$ and those of the linearized DN map. 
\begin{pp}\label{pp_ufour}
    Suppose $K_1,K_2,K_3, K_4$ intersect 4-transversally at a point $q \in \nxxi$ and $s_0>0$ is sufficiently small. 
    Let  $\unionGamma$ be defined in (\ref{def_Gamma}). 
    Suppose $(y, \eta) \in \LcMpo$ is a covector lying along the forward null-bicharacteristic. 
    If $(y, \eta)$ is in $\nxxi\cap \ntxxi$ and away from $\unionGamma$, 
    then we have   
    \begin{align*}
    &{\sigmp}(\mathcal{U}^{(4)})(y, \eta) 
    = 2 (2\pi)^{-3} {\sigmp}({Q}_g)(y, \eta, q, \zeta)  (\zeta_0)^2  \mathcal{C}(\zeta^{(1)}, \zeta^{(2)}, \zeta^{(3)}, \zeta^{(4)}) 
    (\prod_{j=1}^4 {\sigmp}(v_j) (q, \zeta^{(j)})), \nonumber
    \end{align*}          
    where  $\mathcal{C} = \mathcal{C}(\zeta^{(1)}, \zeta^{(2)}, \zeta^{(3)}, \zeta^{(4)})$ is defined in (\ref{eq_C}). 
    Let $(y_|, \eta_|)$ be the projection of $(y, \eta)$ on the boundary.
    Moreover, we have 
    \[
    {\sigmp}(\epslam)(y_|, \eta_|) = 
    \iota  \lge \nu, \eta \rge_g
    \sigmp(\mathcal{R})(y_|, \eta_|, y, \eta){\sigmp}(\mathcal{U}^{(4)})(y, \eta). 
    \]
\end{pp}
\begin{proof}
 There are four different of terms for $\ufour$, see (\ref{u4}). 
 We denote them by $\mathcal{U}^{(4)}_j$, for $j = 1,\ldots, 4$, 
 and analyze each of them.
 Then we compute the principal symbol of $\lge \nu, \nabla \mathcal{U}_j^{(4)} \rge|_{\partial M}$ in $\nxxi \cap \ntxxi$, following the same arguments in \cite[Section 5]{Uhlmann2021a}. 
 We write down these arguments below for completeness. 
 
 In the following, let $(q, \zeta^{(m)}) \in \Lambda_m$ with $m=i,j,k,l$ and $(q, \zeta) \in \Lambda_{q}$ with $\zeta = \zi + \zj + \zk + \zl$. 
 The transversal intersection implies the decomposition of $\zeta$ is unique. 
 
First, we consider 
\[
\mathcal{U}^{(4)}_1 = \sum_{(i,j,k,l) \in \Sigma(4)}  \Qbg(2 \beta_2 \partial_t^2(v_iA_3^{jkl})). 
\]
With the decomposition of $A_3^{ijk} = w_0 + w_1 + w_2 + w_3$ given in (\ref{eq_Aijk}), 
we analyze the singularities of $2 \beta_2 \partial_t^2(v_i w_k)$ for $0 \leq k \leq 3$ in the following and then apply the solution operator $\Qbg$ to each of them. 
Here we replace the notation $\tw_j$ by $w_j$ in the decomposition. 

    Since $w_0 \in I^{3\mu + 2}(\Lambda_{ijk})$, 
    by \cite[Lemma 3.10]{Lassas2018} and \cite[Lemma 4.1]{Wang2019},  
    one can write $2 \beta_2 \partial_t^2(v_iw_0) = w_{0,1} + w_{0,2}$,
    where 
    \[
    w_{0,1} \in I^{4 \mu +5}(\Lambda_q),
    \quad  
    \wfset(w_{0,2}) \subset
    \Lambda_l(\epsilon) \cup \Lambda_{ijk} (\ep).
    \]
    The leading term $w_{0,1}$ has the principal symbol 
    \begin{align*}
    {\sigmp}(w_{0,1})(q, \zeta)
     =(2\pi)^{-3} \mathcal{C}_1(\zeta^{(i)}, \zeta^{(j)}, \zeta^{(k)}, \zeta^{(l)})(\prod_{j=1}^4 {\sigmp}(v_j) (q, \zeta^{(j)})),
    \end{align*}
    at $(q, \zeta) \in \Lambda_q$, where by (\ref{eq_tw0ps}) we have
    \begin{align}\label{C1}
    \mathcal{C}_1(\zeta^{(i)}, \zeta^{(j)}, \zeta^{(k)}, \zeta^{(l)}) 
    = -\frac{2\beta_2(\zeta_0)^2(\ziz + \zjz + \zkz)^2}{|\zi + \zj + \zk|^2_{g^*}} 
    (2\frac{(\zjz + \zkz)^2}{|\zeta^{(j)} + \zeta^{(k)}|^2_{g^*}} \beta^2_2  - \beta_3). 
    \end{align}
    Then we apply $\Qbg$ to $w_{0,1}$ and Proposition \ref{pp_wfb} with its corollary implies 
    \begin{align*}
    &\wfset_b(\Qbg(w_{0,1}))
    \subset \pi_b(\Lambda_q \cup \Lambda_q^{b}) \subset \gsetint \cup \mathcal{H}.  
    \end{align*}
    To find the principal symbol,
    we decompose it as
    $ \Qbg(w_{0,1}) = u^{\text{inc}} + u^{\text{ref}}$ as before,
    where 
    $u^{\text{inc}} $ is the incident wave before the reflection and $u
    ^{\text{ref}}$ is the reflected one. 
    For more details, see the proof of Proposition \ref{pp_Aijk}.  
    Similarly, we have 
    \[
    w^{\text{inc}} \in I^{4 \mu + \frac{7}{2}, -\frac{1}{2}}  ( \Lambda_q ,\Lambda_q^{g} ).
    \]
    If $(y, \eta) \in \partial T^*{ M}$ lies along the forward null-bicharacteristic starting at $(q, \zeta)$, then it belongs to $\Lambda_q^b$ and is the first point where $\Lambda_q^b$ touches the boundary.
    The same argument shows that
        \begin{align*}
    &{\sigmp}(\Qbg(w_{0,1})) (y, \eta) 
    = 
     2 (2\pi)^{-3}    {\sigmp}(\tQ_g)(y, \eta, q, \zeta)   
    \mathcal{C}_1(\zeta^{(i)}, \zeta^{(j)}, \zeta^{(k)}, \zeta^{(l)}) 
    \prod_{j=1}^4 {\sigmp}(v_j) (q, \zeta^{(j)}). \nonumber
    \end{align*}
    when $(y, \eta)$ is away from $\unionGamma$. 
    
    For $\Qbg(w_{0,2})$,  
    away from $\bigcup_{j=1}^4 J^+(\gamma_{x_j, \xi_j}(t_j^b))$, 
    the flow out of $\Lambda_l(\epsilon)$ 
    under the broken bicharacteristic arcs is a neighborhood of $\Lambda_l$, 
    and
    it tends to $\Lambda_l$ 
    when $\ep $ goes to zero.
    Similarly, the flow out of $\Lambda_{ijk}(\ep)$ under the broken bicharacteristic arcs tends to $\Lambda_{ijk}^b$ as $\ep $ goes to $ 0$. 
    
    For $\Qbg(2 \beta_2 \partial_t^2(v_iw_1))$, 
    we use \cite[Lemma 6]{Uhlmann2021a} to show 
    its boundary wave front set tends to a subset of 
    $\pi_b(\Lambda_i \cup \Lambda_{jkl} \cup \Lambda_{jkl}^{b})$, 
    as $\ep$ goes to zero.    
    Here it suffices for us to verify the intersection of $\wfset(v_i)$ and $\wfset(w_1)$ is in $T^*\intM$, by Lemma \ref{lm_intM}. 
    For $\Qbg(2 \beta_2 \partial_t^2(v_iw_2))$, the same argument in \cite[Lemma 6]{Uhlmann2021a} with Lemma \ref{lm_intM} works and its boundary wave front set tends to a subset of
       $\pi_b(\Lambda^{(1)} \cup \Lambda_{jkl} \cup \Lambda_{jkl}^{b})$, as $\epsilon$ goes to zero.     
    For $\Qbg(2 \beta_2 \partial_t^2(v_iw_3))$, 
    we use the same argument as in \cite[Propostion 3.11]{Lassas2018} to conclude 
    its boundary wave front set tends to a subset of 
    $\pi_b(\Lambda^{(1)} \cup \Lambda_{jkl} \cup \Lambda_{jkl}^{b})$ as $\ep$ goes to zero.

     We emphasize that all the analysis above happens in the set $\nxxi \cap \ntxxi$ with sufficiently small $s_0>0$. 
     To summarize, the distribution $\Qbg(2 \beta_2 \partial_t^2(v_i(w_1 + w_2 + w_3)))$ does not have new singularities other than those in 
     \[
     \pi_b(\Lambda^{(1)} \cup \Lambda^{(2)} \cup \Lambda^{(3)} \cup \Lambda^{(3),b}). 
     \]
    Therefore, away from $\unionGamma$, one has 
    \begin{align}\label{eq_psrua}
    &{\sigmp}(\mathcal{U}_1^{(4)} ) (y, \eta) \\
    =&\sum_{(i,j,k,l) \in \Sigma(4)}
      2 (2\pi)^{-3}  
    {\sigmp}(\tQ_g)(y, \eta, q, \zeta)   
    \mathcal{C}_1(\zeta^{(i)}, \zeta^{(j)}, \zeta^{(k)}, \zeta^{(l)}) 
    \prod_{j=1}^4 {\sigmp}(v_j) (q, \zeta^{(j)}). \nonumber
    \end{align}

    Next, 
    recall $\mathcal{U}^{(4)}_2 = \Qbg( \beta_2\partial_t^2(A_2^{ij}A_2^{kl}))$. 
    Away from $\bigcup_{j=1}^4 J^+(\gamma_{x_j, \xi_j}(t_j^b))$, we have 
    \[
     A_2^{ij} \in I^{\mu+1, \mu}(\Lambda_{ij}, \Lambda_i) + I^{\mu+1, \mu}(\Lambda_{ij}, \Lambda_j).
    \]
    The same analysis as in \cite{Kurylev2018, Lassas2018, Wang2019} applies to $\beta_2\partial_t^2(A_2^{ij}A_2^{kl})$. 
    Indeed, by \cite[Lemma 3.8]{Lassas2018}, 
    one has
    \[
    \beta_2\partial_t^2(A_2^{ij}A_2^{kl}) = w_0 + w_1 + w_2,
    \]
    where 
    \begin{align*}
    w_0 \in I^{4 \mu + 5}(\Lambda_q) , 
    \quad \wfset(w_1) \subset
    \Lambda^{(1)} \cup \Lambda^{(2)} \cup \Lambda^{(3)},
    \quad \wfset(w_2) \subset \Lambda^{(1)}(\epsilon).
    \end{align*}
    The principal symbol of $w_0$ is 
    \[
    {\sigmp}(w_0)(q, \zeta) = (2\pi)^{-3}  \mathcal{C}_2(\zeta^{(i)}, \zeta^{(j)}, \zeta^{(k)}, \zeta^{(l)})
    \prod_{j=1}^4 {\sigmp}(v_j) (q, \zeta^{(j)}),
    \]
    with
    \begin{align}\label{C2}
    \mathcal{C}_2(\zeta^{(i)}, \zeta^{(j)}, \zeta^{(k)}, \zeta^{(l)})
    &=-\beta_2^3  (\zeta_0)^2 \frac{(\ziz+\zjz)^2}{| \zi + \zj|^2_{g^*}}\frac{(\zkz+\zlz)^2}{| \zk + \zl|^2_{g^*}}.
    \end{align}
    The same argument as before show that 
    \begin{align*}
    &\wfset_b(\Qbg(w_{0}))
    \subset \pi_b(\Lambda_q \cup \Lambda_q^{b}), 
    \quad \wfset_b(\Qbg(w_{1}))
    \subset \pi_b(\Lambda^{(1)} \cup \Lambda^{(2)} \cup  \Lambda^{(3)} \cup \Lambda^{(3),b}), 
    \end{align*}
    and $\Qbg(w_{2}) $ tends to a subset of 
    $\pi_b(\Lambda^{(1)})$
    when $\ep$ goes to zero. 
    To find the principal symbol of the leading term $\Qbg(w_{0})$, we decompose it as
    $ Q_g(w_{0}) = u^{\text{inc}} + u^{\text{ref}}$ as before.
    Then the same argument shows that 
    if $(y, \eta) \in \partial T^*{ M}$ is away from $\unionGamma$ and lies along the forward null-bicharacteristic starting from $(q, \zeta)$, then we have 
    \begin{align}\label{eq_psrub}
&{\sigmp}(\mathcal{U}_2^{(4)} ) (y, \eta) \\
=& \sum_{(i,j,k,l) \in \Sigma(4)}
2 (2\pi)^{-3}  
{\sigmp}(\tQ_g)(y, \eta, q, \zeta)   
\mathcal{C}_2(\zeta^{(i)}, \zeta^{(j)}, \zeta^{(k)}, \zeta^{(l)}) 
\prod_{j=1}^4 {\sigmp}(v_j) (q, \zeta^{(j)}).\nonumber
    \end{align}
    
    Third, 
    recall $\mathcal{U}^{(4)}_3 = \Qbg(3 \beta_3 \partial_t^2(v_iv_j A_2^{kl}))$.
    By the proof of \cite[Proposition 3.11]{Lassas2018}, see also \cite[Lemma 3.6, Lemma 3.10]{Lassas2018}, 
    one has
    \[
    3 \beta_3 \partial_t^2(v_iv_j A_2^{kl})= w_0 + w_1 + w_2,
    \]
    where 
    \begin{align*}
    &w_0 \in I^{4 \mu + 5}(\Lambda_q) , \quad \wfset(w_1) \subset
    \Lambda^{(1)} \cup \Lambda^{(2)} \cup \Lambda^{(3)}, \quad 
    \wfset(w_2) \subset \Lambda^{(1)}(\ep) \cup \Lambda^{(3)}(\ep). 
    \end{align*}
    The principal symbol for the leading term $w_0$ is 
    \[
    {\sigmp}(w_0)(q, \zeta) = (2\pi)^{-3}  \mathcal{C}_3(\zeta^{(i)}, \zeta^{(j)}, \zeta^{(k)}, \zeta^{(l)})
    \prod_{j=1}^4 {\sigmp}(v_j) (q, \zeta^{(j)}),
    \]
    with
    \begin{align}\label{C3}
    \mathcal{C}_3(\zeta^{(i)}, \zeta^{(j)}, \zeta^{(k)}, \zeta^{(l)}) 
    &=3 \beta_3 \beta_2 (\zeta_0)^2 \frac{(\zkz+\zlz)^2}{| \zk + \zl|^2_{g^*}}.
    \end{align}   
    The same argument shows that 
    if $(y, \eta) \in \partial T^*M$ is away from $\unionGamma$ and lies along the forward null-bicharacteristic starting from $(q, \zeta)$, then we have 
    \begin{align}\label{eq_psruc}
    &{\sigmp}(\mathcal{U}_3^{(4)} ) (y, \eta) \\
    =& \sum_{(i,j,k,l) \in \Sigma(4)}
    2 (2\pi)^{-3}  
    {\sigmp}(\tQ_g)(y, \eta, q, \zeta)   
    \mathcal{C}_3(\zeta^{(i)}, \zeta^{(j)}, \zeta^{(k)}, \zeta^{(l)}) 
    \prod_{j=1}^4 {\sigmp}(v_j) (q, \zeta^{(j)}).\nonumber
    \end{align}
    where $(y_|, \eta_|) $ is the projection of $(y,\eta)$. 
    
    Fourth, 
    recall $\mathcal{U}^{(4)}_4 = \Qbg(\beta_4\partial_t^2(v_iv_jv_kv_l))$.
    By \cite[Lemma 3.8]{Lassas2018}, 
    we write
    \[
    \beta_4\partial_t^2(v_iv_jv_kv_l)
    = w_0 + w_1 + w_2,
    \]
    where
    \begin{align*}
        w_0 \in I^{4 \mu + 5}(\Lambda_q) , 
        \quad \wfset(w_1) \subset
        \Lambda^{(1)} \cup \Lambda^{(2)} \cup \Lambda^{(3)},
        \quad \wfset(w_2) \subset \Lambda^{(1)}(\epsilon).
    \end{align*}
    In this case, the principal symbol for the leading term $w_0$ is 
    \[
    {\sigmp}(w_0)(q, \zeta) = (2\pi)^{-3}  \mathcal{C}_4(\zeta^{(i)}, \zeta^{(j)}, \zeta^{(k)}, \zeta^{(l)}) 
    \prod_{j=1}^4 {\sigmp}(v_j) (q, \zeta^{(j)}).
    \]
    with 
    \begin{align}\label{C4}
    \mathcal{C}_4(\zeta^{(i)}, \zeta^{(j)}, \zeta^{(k)}, \zeta^{(l)}) 
    &= -\beta_4 (\zeta_0)^2.
    \end{align}
    The same argument shows that 
    if $(y, \eta) \in \partial T^*M$ is away from $\unionGamma$ and lies along the forward null-bicharacteristic starting from $(q, \zeta)$, then we have 
    \begin{align}\label{eq_psrud}
    &{\sigmp}(\mathcal{U}_4^{(4)} ) (y, \eta) \\
    = &\sum_{(i,j,k,l) \in \Sigma(4)}
    2 (2\pi)^{-3}  
    {\sigmp}(\tQ_g)(y, \eta, q, \zeta)   
    \mathcal{C}_4(\zeta^{(i)}, \zeta^{(j)}, \zeta^{(k)}, \zeta^{(l)}) 
    \prod_{j=1}^4 {\sigmp}(v_j) (q, \zeta^{(j)}).\nonumber
    \end{align}
    where $(y_|, \eta_|) $ is the projection of $(y,\eta)$. 
   
    Thus,  
    from the analysis above, we combine (\ref{eq_psrua}), (\ref{eq_psrub}), (\ref{eq_psruc}), (\ref{eq_psrud})  to have 
            \begin{align*}
            &{\sigmp}(\mathcal{U}^{(4)})(y, \eta) \\
            = &2 (2\pi)^{-3}   
            {\sigmp}({Q}_g)(y, \eta, q, \zeta) 
            (\zeta_0)^2  \mathcal{C}(\zeta^{(1)}, \zeta^{(2)}, \zeta^{(3)}, \zeta^{(4)}) 
             (\prod_{j=1}^4 {\sigmp}(v_j) (q, \zeta^{(j)})) \nonumber,
            \end{align*}
            where
            \begin{align}\label{eq_C}
            &\mathcal{C}(\zeta^{(1)}, \zeta^{(2)}, \zeta^{(3)}, \zeta^{(4)}) 
             = \sum_{(i,j,k,l) \in \Sigma(4)} (\zeta_0)^{-2} 
            (\mathcal{C}_1+ \mathcal{C}_2 + \mathcal{C}_3 + \mathcal{C}_4)(\zeta^{(i)}, \zeta^{(j)}, \zeta^{(k)}, \zeta^{(l)})  \\
            = & \sum_{(i,j,k,l) \in \Sigma(4)} 
            -(4\frac{(\ziz + \zjz + \zkz)^2}{|\zi + \zj + \zk|^2_{g^*}} + \frac{(\ziz+\zlz)^2}{| \zi + \zl|^2_{g^*}})
            \frac{(\zjz + \zkz)^2}{|\zeta^{(j)} + \zeta^{(k)}|^2_{g^*}}  \beta_2^3 + \nonumber \\
            & \quad \quad \quad 
            + (3 \frac{(\zkz+\zlz)^2}{| \zk + \zl|^2_{g^*}} + 2\frac{(\ziz + \zjz + \zkz)^2}{|\zi + \zj + \zk|^2_{g^*}}) \beta_2 \beta_3
            -\beta_4  \nonumber  
            \end{align}
            with $\mathcal{C}_j$ defined in (\ref{C1}), (\ref{C2}), (\ref{C3}), (\ref{C4}) for $m = 1,2,3,4$.
      By (\ref{eq_LambdaU4}), we have 
      \[
      {\sigmp}(\epslam)(y_|, \eta_|) = \iota  \lge \nu, \eta \rge_g
      \sigmp(\mathcal{R})(y_|, \eta_|, y, \eta){\sigmp}(\mathcal{U}^{(4)})(y, \eta).
      \]     
            
\end{proof}
\section{The recovery of lower order nonlinear terms}\label{sec_lower}
For $k=1,2$, let $p^{(k)}$ solve the boundary value problem (\ref{eq_problem}) with nonlinear terms $F^{(k)}$ that have the expansion in (\ref{eq_nlterm}), i.e.,
\[
\Fk(x,\pk) = \sum_{m=1}^{+\infty} \betak_{m+1}(x) \partial_t^2 (\pk), \quad k = 1, 2,
\]
and satisfy the assumption (\ref{assum_F}).
We suppose
\[
\lambdaFone(f) = \lambdaFtwo(f),
\]
for small boundary data $f$ supported in $(0, T) \times \partial \Omega$.
In this section, we consider the recovery of $\betak_{2}, \betak_{3}, \betak_{4}$ at points
in the suitable larger set
\[
\Wset.
\]
from the fourth order linearization of the DN map.
For convenience, we denote them by lower order nonlinear terms.

Let $q \in \mathbb{W}$ be fixed.
We denote by 
\[
N^\pm(\zeta^o, \varsigma) = \{\zeta \in L_q^{*,\pm}M: \|\zeta-\zeta^o\| < \varsigma\}
\]
a conic neighborhood of a fixed covector $\zeta^o \in L_q^{*,\pm}M$ with small parameter $\varsigma>0$. 
Similarly, we denote the $\varsigma$-neighborhood for a fixed lightlike vector $w \in L^\pm_qM$ by $N^\pm(w, \varsigma)$. 
First, we show that with ${\zh}^{(1)} \in L^*_q M$ given, 
one can perturb ${\zh}^{(1)}$ a little bit to choose lightlike vectors $\zh^{(2)}, \zh^{(3)}, \zh^{(4)}$ such that $\zh^{(j)}, j=1,2,3,4$ are linearly independent and are corresponding to null geodesic segments without cut points.
For convenience, we form the lemma below based on the proof of  \cite[Lemma 3.5]{Kurylev2018}.
\begin{lm}\label{lm_perturb_zeta1}
    Let $q \in \mathbb{W}$ and $\zh^{(1)} \in L^{*,+}_q M$. 
    Suppose there is $(x_1, \xi_1) \in L^{+} V$ with $(q, \zh^{(1)}) = (\gamma_{x_1, \xi_1}(s_1), (\dot{\gamma}_{x_1, \xi_1}(s_1))^b)$ and $0 < s_1 < \rho(x_1, \xi_1)$.
    Then we can find  $\varsigma >0$ such that for any $\zh^{(2)} \in N^+(\zh^{(1)}, \varsigma)$, there exists a vector
    $(x_2, \xi_2) \in L^{+} V$ with $(q, \zh^{(2)}) = (\gamma_{x_2, \xi_2}(s_2), (\dot{\gamma}_{x_2, \xi_2}(s_2))^b)$ and $0 < s_2 < \rho(x_2, \xi_2)$.
    Moreover, one has $(x_1, \xi_1)$ and $(x_2, \xi_2)$ are causally independent.
\end{lm}
\begin{proof}
    First set $w_1 = (\zh^{(1)})^\#$.
    Note that the geodesic $\gamma_{q, -w_1}(s) = \gamma_{x_1, \xi_1}(s_1-s)$ have no cut points for $s \in [0, s_1]$, which implies $s_1 <  \rho(q, -w_1)$.

    We consider all the null geodesics $\gamma_{q, -w}(s)$ that emanate from $q$ in the lightlike direction of $w$, where $w \in N^+(w_1, \varsigma)$ with $\varsigma$ to be specified later.
    By \cite[Theorem 9.33]{Beem2017}, the cut locus function $\rho(q, -w_1)$ is lower semi-continuous on a globally hyperbolic Lorentzian manifold.
    It follows that there is $\varsigma >0 $ such that $s_1 < \rho(q, -w)$ for any $w \in N^+(w_1, \varsigma)$.
    Next, to find $(x_2, \xi_2) \in L^+V$ for such $(q, w)$,
    let $\gamma_{q, -w}(s_x)$ be the intersection points of $\gamma_{q, -w}$ with the Cauchy surface $\{x \in M: \textbf{t}(x) = \textbf{t}(x_1)\}$, where $\textbf{t}$ is the time function on $(M, g)$.
    Indeed, the first cut point comes at or before the first conjugate point along $\gamma_{q, -w_1}(s)$.
    This implies the exponential map is a local diffeomorphism near $-s_1 w_1 \in L^-_q M$, which maps a small neighborhood of $-s_1 w_1$ to a small neighborhood of $x_1$.
    For $w$ close enough to $w_1$, there exist $s_x < s_1 + \epsilon < \rho(q, -w)$ and $x_2 = \gamma_{q,-w}(s_x) \in V$ with $\textbf{t}(x_2) = \textbf{t}(x_1)$,
    for sufficiently small $\ep>0$. 
    Here $\textbf{t}$ is the time function. 
    This proves the statement.
\end{proof}

Now we claim that for any fixed $q \in \mathbb{W}$ and sufficiently small $s_0>0$ one can find \[
(x_j, \xi_j)_{j=1}^4 \subset L^+V, \quad \zeta \in \Lambda_{q}\setminus (\Lambda^{(1)} \cup \Lambda^{(2)} \cup \Lambda^{(3)}), \quad (y, \eta) \in \LcMpo
\]
such that
\begin{itemize}
    \item[(a)] $(x_j, \xi_j)_{j=1}^4$ intersect regularly at $q$ and are causally independent, see (\ref{assump_xj}), 
    \item[(b)] each $\gamma_{x_j, \xi_j}(\mathbb{R}_+)$ hits $\partial M$ exactly once and transversally before it passes $q$,
    \item[(c)] $(y, \eta) \in  \LcMpo$ lies in the  bicharacteristic from $(q, \zeta)$ and additionally there are no cut points along $\gamma_{q, \zeta^\#}$ from $q$ to $y$.
\end{itemize}
Indeed, 
by \cite[Lemma 3.1]{Kurylev2018}, first we pick $\zeta$ and $\hat{\zeta}^{(1)}$ in $L^{*,+}_q M$ such that there exist $(x_1, \xi_1) \in L^{+}V$ and $(\hat{y}, \hat{\zeta}) \in L^{*,+}V$ with
\[
(q, \hat{\zeta}^{(1)}) = (\gamma_{x_1, \xi_1}({s_q}), (\dot{\gamma}_{x_1, \xi_1}(s_q))^b), \quad 
(\hat{y}, \hat{\eta}) = (\gamma_{q, \zeta^\#}(\hat{s}), (\dot{\gamma}_{q, \zeta^\#}(\hat{s}))^b), 
\]
for some $ 0 <s_q< \rho(x_1, \xi_1)$ and 
$ 0 <\hat{s}< \rho(q, \zeta)$.
Next by Lemma \ref{lm_perturb_zeta1}, one can find such three more covectors $\hat{\zeta}^{(j)}$ with $(x_j, \xi_j)$ for $j = 2,3,4$ at $q$ such that $(x_j, \xi_j)_{j=1}^4$ satisfy the condition (a).
To have (b), we can always replace $(x_j, \xi_j)$ by $(\gamma_{x_j, \xi_j}(s_j),\dot{\gamma}_{x_j, \xi_j}(s_j))$ for some $s_j > 0$ if necessary.
Then by \cite[Lemma 2.4]{Hintz2017}, the null geodesic $\gamma_{x_j, \xi_j}(s)$ always hit $\partial M$ transversally before it passes $q$.
To have (c),  recall we have found 
$\zeta \in L_q^{*,+}M$ with
$(\hat{y}, \hat{\eta}) = (\gamma_{q, \zeta^\#}(\hat{s}), (\dot{\gamma}_{q, \zeta^\#}(\hat{s}))^b) \in L^{*, +}V$ for some
$ 0 <\hat{s}< \rho(q, \zeta)$.
We define
\[
s_y = \inf \{ s> 0: \gamma_{q, \zeta}(s) \in \partial M \}, \quad (y, \eta) = (\gamma_{q, \zeta}(s_y), (\dot{\gamma}_{q, \zeta}(s_y)^b).
\]
Note that $s_y < s_{\hat{y}} < \rho(q, \zeta)$.
In addition, the null geodesic $\gamma_{q, \zeta}(s)$ hit $\partial M$ transversally at $y$.
Thus, $(y, \eta) \in \LcMpo$ and (c) is true for $(y, \eta)$.

Moreover, we show in Lemma \ref{lm_perturb_zeta1} that with $\hat{\zeta}^{(1)}$ given, we have freedom to choose $(x_j, \xi_j), j = 2,3,4$, as long as they are from sufficiently small perturbations of $\hat{\zeta}^{(1)}$.
In Lemma \ref{lm_construction} below, we would like to choose special
$(x_j, \xi_j), j = 2,3,4$ such that we can determine the lower order nonlinear terms from a nondegenerate linear system.
Before that, with the construction above, we can apply Proposition \ref{pp_ufour} to have
\[
\mathcal{C}^{(1)}(\zeta^{(1)}, \zeta^{(2)}, \zeta^{(3)}, \zeta^{(4)})
= \mathcal{C}^{(2)}(\zeta^{(1)}, \zeta^{(2)}, \zeta^{(3)}, \zeta^{(4)}),
\]
where $\mathcal{C}^{(k)}$ is defined in (\ref{eq_C}) with $\beta_j$ replaced by $\beta^{(k)}_j$ for $j=2,3,4$.
In the following, we use the notations
\begin{align*}
C(\zeta^{(1)}, \zeta^{(2)}, \zeta^{(3)}, \zeta^{(4)})
&= \sum_{(i,j,k,l) \in \Sigma(4)} (4\frac{(\ziz + \zjz + \zkz)^2}{|\zi + \zj + \zk|^2_{g^*}} + \frac{(\ziz+\zlz)^2}{| \zi + \zl|^2_{g^*}})
\frac{(\zjz + \zkz)^2}{|\zeta^{(j)} + \zeta^{(k)}|^2_{g^*}},\\
D(\zeta^{(1)}, \zeta^{(2)}, \zeta^{(3)}, \zeta^{(4)})  &=  \sum_{(i,j,k,l) \in \Sigma(4)} (3 \frac{(\zkz+\zlz)^2}{| \zk + \zl|^2_{g^*}} + 2\frac{(\ziz + \zjz + \zkz)^2}{|\zi + \zj + \zk|^2_{g^*}}).
\end{align*}
The above analysis shows that from the fourth order linearization of the DN map, one has
\[
-C (\beta^{(1)}_2)^3 + D \beta^{(1)}_2 \beta^{(1)}_3 - \beta^{(1)}_4 =
-C (\beta^{(2)}_2)^3 + D \beta^{(2)}_2 \beta^{(2)}_3 - \beta^{(2)}_4, 
\]
where we write $C$ and $D$ to denote $C(\zeta^{(1)}, \zeta^{(2)}, \zeta^{(3)}, \zeta^{(4)}), D(\zeta^{(1)}$ and $\zeta^{(2)}, \zeta^{(3)}, \zeta^{(4)})$. 
\begin{lm}\label{lm_construction}
For fixed $q \in \mathbb{W}$ and $\zeta, \hat{\zeta}^{(1)} \in L^{*,+}_q M$,
we can find three different sets of nonzero lightlike covectors
\[
(\zeta^{(1),k}, \zeta^{(2),k}, \zeta^{(3),k}, \zeta^{(4),k}), \quad  k = 1,2,3,
\]
such that $\zeta = \sum_{j=1}^4 \zeta^{(j),k}$ with $\zeta^{(1)} = \alpha_j \hat{\zeta}^{(1)}$ for some $\alpha_j$ and the vectors
\[
(-C(\zeta^{(1),k}, \zeta^{(2),k}, \zeta^{(3),k}, \zeta^{(4),k}), D(\zeta^{(1),k}, \zeta^{(2),k}, \zeta^{(3),k}, \zeta^{(4),k}), -1),\ \quad k=1,2,3,
\] are linearly independent.
\end{lm}
This implies that we can solve 
$(\beta^{(1)}_2)^3 - (\beta^{(2)}_2)^3, \beta^{(1)}_2 \beta^{(1)}_3 - \beta^{(2)}_2 \beta^{(2)}_3, \beta^{(1)}_4 - \beta^{(2)}_4$ from a nondegenerate $3 \times 3$ linear system.
\begin{proof}
We choose local coordinates $x = (x^0,x^1, x^2, x^3)$ at $q$ such that $g$ coincides with the Mankowski metric.
By rotating the coordinate system in the spatial variable, without loss of generality, we can assume
\[
\zeta = (-1, 0, \cos \varphi, \sin \varphi), \quad {\zh}^{(1)} = (-1, 1, 0, 0),
\]
where $\varphi \in [0, 2\pi)$. For $\theta \neq 0$ sufficiently small, we choose
\begin{align*}
\zh^{(2)} &= (-1, \cos \theta, \sin \theta \sin \varphi, -\sin \theta \cos \varphi), \\
 \zh^{(3)} &= (-1, \cos \theta, -\sin \theta \sin \varphi, \sin \theta \cos \varphi), \\
\zh^{(4)} &=(-1, \cos \theta, \sin \theta \cos \varphi, \sin \theta \sin \varphi).
\end{align*}
Then we have $\zeta = \sum_{j=1}^4 \zeta^{(j)}$, where we set $\zeta^{(j)} = \alpha_j \zh^{(j)}$ with
\begin{align*}
\alpha_1 = \frac{\cos \theta}{\cos \theta -1} , \quad \alpha_2 = \alpha_3= -\frac{(\cos \theta -1) + \sin \theta}{2(\cos \theta - 1) \sin \theta}, \quad \alpha_4 =  \frac{1}{\sin \theta}.
\end{align*}
One can compute 
\begin{align*}
\frac{\alpha_1}{\alpha_2} = \frac{-2 \cos \theta \cos(\theta/2)}{\costw - \sintw}, \quad 
\frac{\alpha_1}{\alpha_4} = \frac{\cos \theta \cos(\theta/2)}{ - \sintw}, \quad 
\frac{\alpha_2}{\alpha_4} = \frac{\costw - \sintw}{2\sintw}.
\end{align*}
Note $\zeta_0^{(j)} = -\alpha_j, j = 1, 2,3, 4$.
We compute
\[
| \zk + \zl|^2_{g^*} = 2 \lge \zk, \zl \rge_{g^*} = 2 \alpha_k \alpha_l \lge \zh^{(k)}, \zh^{(l)} \rge_{g^*},
\]
which implies
\begin{align*}
&| \zeta^{(1)} + \zeta^{(2)}|^2_{g^*} = 2 \alpha_1\alpha_2 (\cos \theta - 1),
&| \zeta^{(1)} + \zeta^{(3)}|^2_{g^*} = 2 \alpha_1\alpha_3(\cos \theta - 1), \\
&| \zeta^{(1)} + \zeta^{(4)}|^2_{g^*} = 2 \alpha_1\alpha_4(\cos \theta - 1),
&| \zeta^{(2)} + \zeta^{(3)}|^2_{g^*} = -4\alpha_2 \alpha_3 \sin^2 \theta, \\
&| \zeta^{(2)} + \zeta^{(4)}|^2_{g^*} = -2\alpha_2 \alpha_4 \sin^2 \theta,
&| \zeta^{(3)} + \zeta^{(4)}|^2_{g^*} = -2\alpha_3 \alpha_4 \sin^2 \theta,
\end{align*}
It follows that
\begin{align*}
&S_{12} \equiv \frac{(\zeta_0^{(1)} + \zeta_0^{(2)})^2}{| \zeta^{(1)} + \zeta^{(2)}|^2_{g^*}}  = \frac{(\alpha_1 + \alpha_2)^2}{2 \alpha_1\alpha_2 (\cos \theta - 1)}
\equiv \frac{1}{2(\cos \theta -1)}I_1,\\
&S_{13} \equiv  \frac{(\zeta_0^{(1)} + \zeta_0^{(3)})^2}{| \zeta^{(1)} + \zeta^{(3)}|^2_{g^*}}  = \frac{(\alpha_1 + \alpha_3)^2}{2 \alpha_1\alpha_3 (\cos \theta - 1)}
\equiv \frac{1}{2(\cos \theta -1)}I_1, \\
&S_{14} \equiv \frac{(\zeta_0^{(1)} + \zeta_0^{(4)})^2}{| \zeta^{(1)} + \zeta^{(4)}|^2_{g^*}}  = \frac{(\alpha_1 + \alpha_4)^2}{2 \alpha_1\alpha_4 (\cos \theta - 1)}
\equiv \frac{1}{2(\cos \theta -1)}I_3, \\
&S_{23} \equiv \frac{(\zeta_0^{(2)} + \zeta_0^{(3)})^2}{| \zeta^{(2)} + \zeta^{(3)}|^2_{g^*}}
= \frac{(\alpha_2 + \alpha_3)^2}{-4\alpha_2 \alpha_3 \sin^2 \theta}
= -\frac{1}{\sin^2 \theta}, \\
&S_{24} \equiv \frac{(\zeta_0^{(2)} + \zeta_0^{(4)})^2}{| \zeta^{(2)} + \zeta^{(4)}|^2_{g^*}}  = \frac{(\alpha_2 + \alpha_4)^2}{-2\alpha_2 \alpha_4 \sin^2 \theta}
\equiv -\frac{1}{2\sin^2 \theta}I_2,\\
&S_{34} \equiv \frac{(\zeta_0^{(3)} + \zeta_0^{(4)})^2}{| \zeta^{(3)} + \zeta^{(4)}|^2_{g^*}}  = \frac{(\alpha_3 + \alpha_4)^2}{-2\alpha_3 \alpha_4 \sin^2 \theta}
\equiv -\frac{1}{2\sin^2 \theta}I_2,
\end{align*}
where we write $s = \sintw$ and compute
\begin{align*}
I_1 &= \frac{(\alpha_1 + \alpha_2)^2}{\alpha_1\alpha_2 }
= -\frac{1}{2 \cos \theta} + (-2 \costw + \frac{1}{2 \costw \costt})s + 2s^2,\\
I_2 &= \frac{(\alpha_2 + \alpha_4)^2}{\alpha_2\alpha_4 }
= (\frac{\costw}{2})\frac{1}{s} +\frac{3}{2} + 2 \costw s + 2 s^2 +  (-\frac{4}{\sintw - \costw}) s^3,\\
I_3 &= \frac{(\alpha_1 + \alpha_4)^2}{\alpha_1\alpha_4 }
= (-\costw \cos \theta) \frac{1}{s}+ 2 + (-\frac{1}{\costw\cos \theta})s.
\end{align*}
Here we expand each term according to the order of $s$,
in order to analyze its behavior near $\theta = 0$.
If we further write
\begin{align*}
&\cos \theta  = 1- 2s^2, \quad
\frac{1}{\costt} = 1 + 2s^2 + 4s^4 + \mathcal{O}(s^6), \\
&\costw = 1- \frac{1}{2}s^2 - \frac{1}{8}s^4 + \mathcal{O}(s^6),\quad
\frac{1}{\costw} = 1 + \frac{s^2}{2} + \frac{3}{8} s^4 + \mathcal{O}(s^6), \\
&{\costw \costt} = 1 - \frac{5}{2}s^2 + \mathcal{O}(s^4),\quad
\frac{1}{\costw \costt} = 1 + \frac{5}{2}s^2 + \mathcal{O}(s^4),
\end{align*}
then using Mathematica codes, we have
\begin{align*}
&I_1 = -\frac{1}{2} - \frac{3}{2}s + s^2 + \frac{9}{4}s^3 - 2s^4 + \mathcal{O}(s^5), \\
&I_2 = \frac{1}{2s} + \frac{3}{2}  + \frac{7}{4}s + 2s^2 + \frac{47}{16}s^3 + 4s^4 + \mathcal{O}(s^5), \\
&I_3 = -\frac{1}{s} + 2  + \frac{3}{2}s - \frac{27}{8}s^3 + \mathcal{O}(s^5).
\end{align*}
Plugging it back, we have 
\begin{align*}
&S_{12} = S_{13} = \frac{1}{8s^2} + \frac{3}{8s} -\frac{1}{4}- \frac{9}{16}s + \frac{1}{2}s^2 + \mathcal{O}(s^3), \\
&S_{23} = -\frac{1}{4s^2}  -\frac{1}{4}-  \frac{1}{4}s^2 + \mathcal{O}(s^4), \\
&S_{24} = S_{34} =-\frac{1}{16s^3} - \frac{3}{16s^2} - \frac{9}{32s} - \frac{83}{128}s + \mathcal{O}(s^2), \\
&S_{14} = \frac{1}{4s^3}-\frac{1}{2s^2} - \frac{3}{8s} +\frac{27}{32}s + \mathcal{O}(s^3).
\end{align*}
It follows that
\begin{align*}
&D_1(\zeta^{(1)}, \zeta^{(2)}, \zeta^{(3)}, \zeta^{(4)})  =  \sum_{(i,j,k,l) \in \Sigma(4)} 3 \frac{(\zkz+\zlz)^2}{| \zk + \zl|^2_{g^*}}\\
&= 3 \times 4 (S_{12}\times 2 + S_{14} + S_{23} + S_{24}\times 2)
=\frac{3}{2s^3} - \frac{21}{2s^2} - \frac{9}{4s} + \mathcal{O}(1).
\end{align*}
Since for $(i,j,k,l) \in \Sigma(4)$ we have
\[
|\zj + \zk + \zl|^2_{g^*} = |\zeta - \zi|^2_{g^*} = -2 \alpha_i \lge \zeta, \zh^{(i)} \rge_{g^*},
\]
then
\begin{align*}
&| \zeta^{(1)} + \zeta^{(2)}+ \zeta^{(3)}|^2_{g^*}
= -2 \alpha_4(\sin \theta -1),
&| \zeta^{(1)} + \zeta^{(2)}+ \zeta^{(4)}|^2_{g^*}
= 2\alpha_3,\\
&| \zeta^{(1)} + \zeta^{(3)}+ \zeta^{(4)}|^2_{g^*}
= 2\alpha_2,
&| \zeta^{(2)} + \zeta^{(3)}+ \zeta^{(4)}|^2_{g^*}
= 2\alpha_1.
\end{align*}
We compute
\begin{align*}
&R_{123} \equiv \frac{(\zeta_0^{(1)} + \zeta_0^{(2)}+ \zeta_0^{(3)})^2}{| \zeta^{(1)} + \zeta^{(2)}+ \zeta^{(3)}|^2_{g^*}}
= \frac{(1-\alpha_4)^2}{-2 \alpha_4(\sin \theta -1)}
=\frac{1-\sin\theta}{2\sin \theta}
= -\frac{1}{2} + \frac{1}{4 \costw} s,\\
&R_{124} \equiv \frac{(\zeta_0^{(1)} + \zeta_0^{(2)}+ \zeta_0^{(4)})^2}{| \zeta^{(1)} + \zeta^{(2)}+ \zeta^{(4)}|^2_{g^*}}
= \frac{(1-\alpha_3)^2}{2\alpha_3}
= -\frac{1}{4}(\frac{1}{\sin \theta} + \frac{1}{\cos \theta -1}) -1 -\frac{(\cos \theta -1 )\sin \theta}{(\cos \theta -1 ) + \sin \theta}\\
&\quad \quad \quad \quad \quad \quad \quad \quad\quad \quad \quad \quad \quad
=  \frac{1}{8s^2} - \frac{1}{8 \costw} \frac{1}{s} -1 + (1 + \frac{1}{\costt}) s^2 + (\frac{2\costw}{\costt})s^3,\\
&R_{134} \equiv \frac{(\zeta_0^{(1)} + \zeta_0^{(3)}+ \zeta_0^{(4)})^2}{| \zeta^{(1)} + \zeta^{(3)}+ \zeta^{(4)}|^2_{g^*}}  = \frac{(1-\alpha_2)^2}{2\alpha_2}
= R_{124}, \\
&R_{234} \equiv \frac{(\zeta_0^{(2)} + \zeta_0^{(3)}+ \zeta_0^{(4)})^2}{|\zeta^{(2)} + \zeta^{(3)}+ \zeta^{(4)}|^2_{g^*}}
= \frac{(1-\alpha_1)^2}{2 \alpha_1}
=\frac{1}{2 \cos \theta(\cos \theta -1)}
= -\frac{1}{4 \costt} \frac{1}{s^2}.
\end{align*}
This implies that
\begin{align*}
&R_{123} = R_{134}
=  \frac{1}{4s} -\frac{1}{2}+ \frac{1}{8}s +  \mathcal{O}(s^3),\\
&R_{124} = \frac{1}{8s^2} - \frac{1}{8s} -1 - \frac{1}{16}s + 2s^2 + \mathcal{O}(s^3),\\
&R_{234}
 = -\frac{1}{4s^2} - \frac{1}{2} - s^2 + \mathcal{O}(s^3).
\end{align*}
Then we have
\begin{align*}
D_2(\zeta^{(1)}, \zeta^{(2)}, \zeta^{(3)}, \zeta^{(4)})
&=  \sum_{(i,j,k,l) \in \Sigma(4)}  2\frac{(\ziz + \zjz + \zkz)^2}{|\zi + \zj + \zk|^2_{g^*}}
= - 36 + \mathcal{O}({s}).
\end{align*}
Thus, one has
\begin{align}\label{eq_d}
d(\theta) &\equiv D(\zeta^{(1)}, \zeta^{(2)}, \zeta^{(3)}, \zeta^{(4)}) = D_1(\zeta^{(1)}, \zeta^{(2)}, \zeta^{(3)}, \zeta^{(4)}) + D_2(\zeta^{(1)}, \zeta^{(2)}, \zeta^{(3)}, \zeta^{(4)}) \nonumber \\
&=\frac{3}{2s^3} - \frac{21}{2s^2} - \frac{9}{4s} + \mathcal{O}(1).
\end{align}
Next, we compute
\begin{align*}
&C_2(\zeta^{(1)}, \zeta^{(2)}, \zeta^{(3)}, \zeta^{(4)})
= \sum_{(i,j,k,l) \in \Sigma(4)} \frac{(\ziz+\zlz)^2}{| \zi + \zl|^2_{g^*}}
\frac{(\zjz + \zkz)^2}{|\zeta^{(j)} + \zeta^{(k)}|^2_{g^*}}\\
=& 8(S_{12}S_{34} \times 2 + S_{14}S_{23})
=-\frac{5}{8s^5} + \frac{1}{4s^4} - \frac{19}{16s^3} -\frac{1}{4s^2}- \frac{195}{64s} + \mathcal{O}(1).
\end{align*}
We have
\begin{align*}
&C_1(\zeta^{(1)}, \zeta^{(2)}, \zeta^{(3)}, \zeta^{(4)})
= \sum_{(i,j,k,l) \in \Sigma(4)} 4\frac{(\ziz + \zjz + \zkz)^2}{|\zi + \zj + \zk|^2_{g^*}}
\frac{(\zjz + \zkz)^2}{|\zeta^{(j)} + \zeta^{(k)}|^2_{g^*}}
=  C_{11} + C_{12} + C_{13},
\end{align*}
where
\begin{align*}
C_{11} = 4 \times 2 R_{123}(S_{12}\times 2 + S_{23})
= \frac{3}{2s^2} - \frac{9}{2s} + \mathcal{O}(1),
\end{align*}
and
\begin{align*}
C_{12} =4 \times 2R_{234}(S_{24}\times 2 + S_{23})
= \frac{1}{4s^5} + \frac{5}{4s^4} + \frac{13}{8s^3} +\frac{19}{4s^2} + \frac{187}{32s} + \mathcal{O}(1),
\end{align*}
and
\begin{align*}
C_{13} = 4 \times 4 R_{124}(S_{12} + S_{24} + S_{14})
= \frac{3}{4s^5} - \frac{3}{2s^4} - \frac{39}{16s^3} +\frac{8}{s^2} + \frac{749}{64s} + \mathcal{O}(1).
\end{align*}
It follows that
\begin{align*}
C_1(\zeta^{(1)}, \zeta^{(2)}, \zeta^{(3)}, \zeta^{(4)})
= \frac{5}{8s^5} - \frac{1}{4s^4} - \frac{13}{16s^3} +\frac{57}{4s^2} + \frac{835}{64s} + \mathcal{O}(1).
\end{align*}
Then we have
\begin{align}\label{eq_c}
&c(\theta) \equiv  C(\zeta^{(1)}, \zeta^{(2)}, \zeta^{(3)}, \zeta^{(4)})
= -\frac{2}{s^3} +\frac{14}{s^2} + \frac{10}{s} + \mathcal{O}(1).
\end{align}
Combining (\ref{eq_c}) and (\ref{eq_d}), we have
\begin{align}\label{eq_limitcd}
c(\theta) + \frac{4}{3} d(\theta) = \frac{7}{s} + \mathcal{O}(1).
\end{align}

We claim that there exist sufficiently small $\theta_1, \theta_2, \theta_3 > 0 $ such that
\[
(c(\theta_j), -d(\theta_j), 1), \quad j=1,2,3,
\]
are linearly independent.
    Indeed,
    we define \[
    g(\theta) \equiv  c(\theta) + \frac{4}{3} d(\theta).
    \]
    The analysis above shows that $g(\theta) = {7}/{s} + \mathcal{O}(1)$ for $s = \sintw$ small enough.
    In the following, we would like to find $\theta_j, j=1,2,3$ in a small neighborhood of $\theta = 0$, such that the determinant does not vanish, i.e.,
    \begin{align*}
    &\det
    \left[ {
        \begin{array}{ccc}
        c(\theta_1) & d(\theta_1) & 1 \\
        c(\theta_2) & d(\theta_2) & 1 \\
        c(\theta_3) & d(\theta_3) & 1
        \end{array}
    } \right] \neq 0
    \end{align*}
    For simplification, we write $g(\theta_j) = g_j$ and $d(\theta_j) = d_j$ for $j = 1,2,3$.
    We compute
    \begin{align*}
    &\det
    \left[ {
        \begin{array}{ccc}
        c(\theta_1) & d(\theta_1) & 1 \\
        c(\theta_2) & d(\theta_2) & 1 \\
        c(\theta_3) & d(\theta_3) & 1
        \end{array}
    } \right]
    =
    \det
    \left[ {
        \begin{array}{ccc}
        g_1 & d_1& 1 \\
        g_2 & d_2 & 1 \\
        g_3 & d_3 & 1
        \end{array}
    } \right]\\
    & = (g_2 d_3 - g_3 d_2) - (g_1 d_3 - g_3 d_1) + (g_1 d_2 - g_2 d_1)\\
    &= (g_2 - g_3)(d_2- d_1) + (g_1 - g_2)(d_2 - d_3).
    \end{align*}
    We choose sufficiently small $\theta_1$ with
    $
    \sin \theta_2 = r \sin \theta_1, \sin \theta_3 ={r^2} \sin \theta_1
    $
    for small parameter $0< r < 1$.
    Recall $s = \sin (\theta_1/2)$.
    By (\ref{eq_d}) and (\ref{eq_limitcd}), one has
    \begin{align*}
    &d_2 = - \frac{3}{ 2\sin^3 (\theta_2/2)} + \mathcal{O}(\frac{1}{ \sin^2 (\theta_2/2)})
    = - \frac{3 }{ 2s^3r^3} + \mathcal{O}(\frac{1}{ s^2r^2}), \\
    &d_3 = - \frac{3}{ 2\sin^3 (\theta_3/2)} + \mathcal{O}(\frac{1}{ \sin^2 (\theta_3/2)})
    = - \frac{3 }{ 2s^3r^6} + \mathcal{O}(\frac{1}{ s^2r^4}),\\
    & g_2 = \frac{7}{\sin (\theta_2/2)}+ \mathcal{O}(1) = \frac{7}{ s r}+ \mathcal{O}(1), \quad
    g_3 = \frac{7}{\sin (\theta_3/2)}+ \mathcal{O}(1) = \frac{7}{ s r^2}+ \mathcal{O}(1). \\
    \end{align*}
    This implies that
    \begin{align*}
    d_2 - d_1 = - \frac{3 }{ 2s^3r^3} + \mathcal{O}(\frac{1}{ s^2r^2} + \frac{1}{ s^3}), \quad
    &d_2 - d_3 = - \frac{3 }{ 2s^3r^6} + \mathcal{O}(\frac{1}{ s^2r^4} + \frac{1}{ t^3}),\\
    g_2 - g_3 = -\frac{7}{ s r^2}+ \mathcal{O}(\frac{1}{ s r}), \quad
    &g_1 - g_2 = -\frac{7}{ s r}+ \mathcal{O}(\frac{1}{ s}).
    \end{align*}
    Thus, we have
    \begin{align*}
    (g_2 - g_3)(d_3- d_1) + (g_1 - g_2)(d_2 - d_3) = \frac{21 }{ 2s^4 r^7} + \mathcal{O}(\frac{1}{ s^4r^6}) \neq 0.
    \end{align*}
    for sufficiently small $s$ and $r$.
    This proves the desired result.
\end{proof}
Thus, combining Lemma \ref{lm_perturb_zeta1}, Lemma \ref{lm_construction}, and Proposition \ref{pp_ufour}, we can choose three different sets
\[
(\zeta^{(1),j}, \zeta^{(2),j}, \zeta^{(3),j}, \zeta^{(4),j}), \quad k = 1,2,3,
\]
such that from
\[
-C^{(j)} ((\beta^{(1)}_2)^3- (\beta^{(2)}_2)^3) + D^{(j)} (\beta^{(1)}_2 \beta^{(1)}_3- \beta^{(2)}_2 \beta^{(2)}_3) - (\beta^{(1)}_4 - \beta^{(2)}_4) =0, \quad j=1,2,3,
\]
we can conclude that
\[
(\beta^{(1)}_2)^3- (\beta^{(2)}_2)^3 = \beta^{(1)}_2 \beta^{(1)}_3- \beta^{(2)}_2 \beta^{(2)}_3 = \beta^{(1)}_4 - \beta^{(2)}_4 = 0.
\]
Here we write $C^{(j)} = C(\zeta^{(1),j}, \zeta^{(2),j}, \zeta^{(3),j}, \zeta^{(4),j})$ and $D^{(j)} = D(\zeta^{(1),j}, \zeta^{(2),j}, \zeta^{(3),j}, \zeta^{(4),j})$.
It follows that
\[\beta^{(1)}_2 = \beta^{(2)}_2, \quad \beta^{(1)}_4 = \beta^{(2)}_4
\]
at $q$.
In the case that $\beta^{(1)}_2 = \beta^{(2)}_2$ does not vanish,
we conclude that $\beta^{(1)}_3 = \beta^{(2)}_3$.
Otherwise, if $\beta^{(1)}_2(q) = \beta^{(2)}_2(q) = 0$,  one can use the third order linearization
\[
\partial_{\epsilon_1}\partial_{\epsilon_2}\partial_{\epsilon_3} \Lambda_F (f) |_{\epsilon_1 = \epsilon_2 = \epsilon_3=0} = \lge \nu, \nabla \mathcal{U}^{(3)} \rge|_{\partial M}
\]
to show that
$\beta^{(1)}_3 = \beta^{(2)}_3$.
More explicitly, by Proposition \ref{pp_Aijk} and a similar construction of $\zeta^{(1)}, \zeta^{(2)}, \zeta^{(3)}$ in \cite{Hintz2020}, we have
\[
\mathcal{Q}^{(1)}(\zeta^{(1)}, \zeta^{(2)}, \zeta^{(3)}) = \mathcal{Q}^{(2)}(\zeta^{(1)}, \zeta^{(2)}, \zeta^{(3)})
\quad \Rightarrow \quad
\beta^{(1)}_3 = \beta^{(2)}_3.
\]
In both cases, we have
\[
\beta^{(1)}_i = \beta^{(2)}_i, \quad \text{ for }i = 2,3,4,
\]
at any $q \in \mathbb{W}$.


\section{Piriou Conormal distributions}\label{sec_Piriou}
In the following, we introduce Piriou Conormal distributions.
In Section \ref{sec_N}, we determine $\beta_N$ at each $q \in \mathbb{W}$ from the singularities of the $N$th order linearization of the DN map by using this special class of conormal distributions.

Let $K$ be a submanifold of $M$ with $\codim(K)=1$.
We follow \cite{Piriou1988} to define the class of conormal distributions that vanish to certain order at $K$.
Note the order of conormal distributions in \cite{Piriou1988} is the same as that of symbols and we shift it by $-\frac{n}{4} + \frac{N}{2}$ following the definition in \cite{MR2304165}.
\begin{df}
    Let $ m < -1$ and  $k(m) \in \mathbb{N}$ such that $ -m -2 \leq k(m) < -m -1$.
    We say $u \in \Ir^{m - \frac{n}{4} + \frac{1}{2}}(K)$ if $u \in \Ical^{m - \frac{n}{4} + \frac{1}{2}}(K)$ vanishing to order $k(m) + 1$ at $K$.
\end{df}
By \cite[Propostion 2.3 and 2.4]{Piriou1988}, a distribution $u \in \Ir^{m}(K)$ if and only if there exists $h \in C^\infty$ vanishing to order $k(m)$ such that $u = hv$ with $v \in \Ir^{m + k(m)}(K)$, for $m <-1$.
Additionally, by \cite{Barreto2021a}, for any $u \in \Ical^{m - \frac{n}{4} + \frac{1}{2}}(K)$ with compact support,
by subtracting a compactly supported smooth function whose derivatives at $K$ up to order $k(m) +1$ coincide with those of $u$, one can modify $u$ such that $u \in  \Ir^{m - \frac{n}{4} + \frac{1}{2}}(K)$.
This can be done since we can show that $u$ is continuous up to order $k(m) +1$ at $K$, for $m<-1$.

\begin{lm}[{\cite[Proposition 2.4]{Barreto2021a}}]\label{lm_v1}
Let $m <-1$ and $u \in \Ical^{m - \frac{n}{4} + \frac{1}{2}}(K)$.
Then there exists $e \in C^\infty$ such that $u = v + e$ with $v \in \Ir^{m - \frac{n}{4} + \frac{1}{2}}(K)$.
\end{lm}

%

Moreover, we have the following lemma, see \cite[Theorem 2.1]{Piriou1988}
and also \cite[Proposition 2.5]{Barreto2021a}.
\begin{lm}\label{cl_Piriou}
Let $m\in \mathbb{N}$ with $m \geq 2$.
Suppose $v_1 \in \Ir^\mu(\Lambda_1)$.
Then
\[
v_1^m \in \Ir^{\mu + (m -1)(\mu+\frac{3}{2})}(\Lambda_1),
\]
with the principal symbol given by $m$-fold fiberwise convolution
\begin{align*}
\sigmp(v_1^m)
= (2\pi)^{-(m-1)} \underbrace{{\sigmp}(v_1) \ast \ldots \ast {\sigmp}(v_1) }_{m}.
\end{align*}
\end{lm}
Now we can prove the following lemma about $v_1^m v_2 v_3v_4$.
\begin{lm}\label{lm_psmodel_v}
    Let $m\in \mathbb{N}$ with $m \geq 2$.
    Suppose $K_1,K_2,K_3, K_4$ intersect 4-transversally at a point $q \in \nxxi$. 
    Suppose $v_1 \in \Ir^{\mu} (\Lambda_1)$ and $v_j \in \Ical^{\mu} (\Lambda_j)$ for $j = 2, 3, 4$.
    Let $\Xi_4 \coloneqq \Lambda^{(1)} \cup \Lambda^{(2)} \cup \Lambda^{(3)}$.
    Then microlocally
    away from $\Xi_4$, we have
    \[
    v_1^m v_2 v_3v_4 \in \Ical^{4\mu + 3 +  (m-1)(\mu+ \frac{3}{2})}(\Lambda_q).
    \]
    Moreover, for $(q, \zeta) \in \Lambda_q \setminus \Xi_4$,  the principal symbol is given by
    \begin{align}\label{eq_psmodel_v}
     \sigmp(v_1^m v_2 v_3v_4 )(q, \zeta) = (2\pi)^{-3} \sigmp(v_1^m)(q, \zeta^{(1)}) \prod_{j=2}^{4} \sigmp(v_j)(q, \zj),
    \end{align}
    where the decomposition $\zeta = \sum_{j=1}^4 \zj$ with $\zj \in \Lambda_j$ is unique.
    Note that $\sigmp(v_1^{m})$ is  given by the $m$-fold fiberwise convolution in Lemma \ref{cl_Piriou}.
    \begin{proof}
        By Lemma \ref{cl_Piriou}, we have
        \[
        {v}_1^{m} \in \Ir^{\mu + (m-1)(\mu+\frac{3}{2})}(M; \Lambda_1)
        \]
        with $\sigmp(v_1^{m})$ given by the $m$-fold fiberwise convolution.
        Then by \cite[Lemma 3.3 and Lemma 3.8]{Lassas2018}, we have
        \[
        {v}_1^{m} v_2 v_3 v_4 \in \Ical^{4 \mu + 3 +  (m-1)(\mu+\frac{3}{2})}(M; \Lambda_q)
        \]
        microlocally away $\Xi_4$ and the principal symbol given in (\ref{eq_psmodel_v}).
    \end{proof}
\end{lm}
Combining Lemma \ref{lm_psmodel_v} with the same arguments in Section \ref{Sec_threewaves} and Section \ref{Sec_fourwaves}, we have the following lemma.

\begin{lm}\label{lm_psmodel}
    Suppose $K_1,K_2,K_3, K_4$ intersect 4-transversally at a point $q \in \nxxi$.
    Let $\unionGamma$ be defined in (\ref{def_Gamma}).
    Let $v_1 \in \Ir^{\mu} (\Lambda_1)$ and $v_j \in \Ical^{\mu} (\Lambda_j)$ for $j = 2, 3, 4$.
    Suppose $(y, \eta) \in \LcMpo$ is a covector lying along the forward null-bicharacteristic of $\sq$ starting at $(q, \zeta) \in \Lambda_q$.
    If $(y, \eta)$ is contained in $\nxxi \cap \ntxxi$ and away from $\unionGamma$,
    then with sufficiently small $s_0>0$ we have
    \begin{align*}
    &{\sigmp}(\Qb_g(\partial_t^2({v}_1^{m} v_2 v_3 v_4)))(y, \eta) \\
    =& - 2 (2\pi)^{-(m+1)} {\sigmp}({Q}_g)(y, \eta, q, \zeta)  (\zeta_0)^2
    \sigmp(v_1^m)(q, \zeta^{(1)})
    \prod_{j=2}^{4} \sigmp(v_j)(q, \zj), \nonumber
    \end{align*}
    where the decomposition $\zeta = \sum_{j=1}^4$ with $\zj \in \Lambda_j$ is unique.
%
    Note the homogeneous term $\sigmp(v_1^m)$ is given by the $m$-fold fiberwise convolution in  Lemma \ref{cl_Piriou}.
\end{lm}

\section{Recovery of the higher order nonlinearity}\label{sec_N}
For $k=1,2$, let $p^{(k)}$ solve the boundary value problem (\ref{eq_problem}) with nonlinear terms $F^{(k)}$ that satisfies the expansion in (\ref{eq_nlterm}), i.e.,
\[
\Fk(x, \pk,\partial_t \pk, \partial^2_t \pk) = \sum_{m=1}^{+\infty} \beta^{(k)}_{m+1}(x) \partial_t^2 ((\pk)^{m+1}), \quad k = 1, 2,
\]
Suppose
\[
\Lambda^{(1)}_{F^{(1)}}(f) = \Lambda^{(2)}_{F^{(2)}}(f),
\]
for small boundary data $f$ supported in $(0, T) \times \partial \Omega$.
In this section, we consider the recovery of $\beta^{(k)}_{m}, m \geq 5$
in the suitable larger set
\[
\Wset
\]
from the $m$-th order linearization of the DN map.
For convenience, we denote them by higher order nonlinear terms.
The analysis in Section \ref{sec_lower} shows one can recover $\beta^{(k)}_2, \beta^{(k)}_3, \beta^{(k)}_4$ from the fourth order linearization.

For fixed $q \in \mathbb{W}$, we consider the same construction as in Section \ref{sec_lower}, i.e.,
\[
(\vec{x}, \vec{\xi})_{j=1}^4 \subset L^+V, \quad \zeta \in \Lambda_{q}\setminus (\Lambda^{(1)} \cup \Lambda^{(2)} \cup \Lambda^{(3)}), \quad (y, \eta) \in \LcMpo
\]
such that
\begin{itemize}
    \item[(a)] $(\vec{x}, \vec{\xi})_{j=1}^4$ intersect regularly at $q$ and are causally independent, see (\ref{assump_xj}), 
    \item[(b)] each $\gamma_{x_j, \xi_j}(\mathbb{R}_+)$ hits $\partial M$ exactly once and transversally before it passes $q$,
    \item[(c)] $(y, \eta) \in  \LcMpo$ lies in the  bicharacteristic from $(q, \zeta)$ and additionally there are no cut points along $\gamma_{q, \zeta^\#}$ from $q$ to $y$.
\end{itemize}
For $(\vec{x}, \vec{\xi})_{j=1}^4$ satisfying the conditions above, we construct $\Sigma_j, K_j, \Lambda_j, v_j, f_j, f$
and define
\begin{align*}
\mathcal{U}^{(N,k)}
= \partial^{N-3}_{\epsilon_1}\partial_{\epsilon_2}\partial_{\epsilon_3}\partial_{\epsilon_4} p^{(k)} |_{\epsilon_1 = \epsilon_2 = \epsilon_3 = \ep_4=0}, \quad k=1,2,
\end{align*}
where $p^{(k)}$ solves the boundary value problem (\ref{eq_problem}).
In particular, we choose $v_1 \in \Ir^\mu(\Lambda_1)$ as in Section \ref{sec_Piriou}.
We write
\[
\mathcal{U}^{(N,k)} = \mathcal{U}^{(N,k)}_{N} + \mathcal{Q}_N(\beta^{(k)}_2, \beta^{(k)}_3, \ldots, \beta^{(k)}_{N_1}),
\]
where
\[
\mathcal{U}^{(N,k)}_{N}  = \partial^{N-3}_{\epsilon_1}\partial_{\epsilon_2}\partial_{\epsilon_3}\partial_{\epsilon_4} ({\Qb_g}(\beta_N \partial_t^2(v^N))) |_{\epsilon_1 = \epsilon_2 = \epsilon_3 = \ep_4=0},
\]
and $\mathcal{Q}_{N}(\beta^{(k)}_2, \beta^{(k)}_3, \ldots, \beta^{(k)}_{N_1})$ contains the terms involved with $\beta^{(k)}_2, \ldots, \beta^{(k)}_{N-1}$, see Section \ref{subsec_assyp}.

We note that one has
\[
\partial_{\epsilon_1}^{N-3}\partial_{\epsilon_2}\partial_{\epsilon_3}\partial_{\epsilon_4} \Lambda^{(k)}_{F^{(k)}}(f) |_{\epsilon_1 = \epsilon_2 = \epsilon_3 = \ep_4=0}
= \lge \nu, \nabla \mathcal{U}^{(N,k)} \rge|_{(0,T) \times \partial \Omega}
, \quad k = 1,2.
\]
In the following we use an induction procedure to determine $\beta^{(k)}_N$.
Assuming $\beta^{(k)}_j$ has been determined for $j <N$,
we subtract the contribution of $\mathcal{Q}_N(\beta^{(k)}_2, \beta^{(k)}_3, \ldots, \beta^{(k)}_{N-1})$ to
obtain that
\begin{align}\label{eq_UN}
\lge \nu, \nabla \mathcal{U}^{(N,1)}_N \rge|_{(0,T) \times \partial \Omega} =
\lge \nu, \nabla \mathcal{U}^{(N,2)}_N \rge|_{(0,T) \times \partial \Omega}.
\end{align}
For $k=1,2$,
we compute
\[
\mathcal{U}^{(N,k)}_{N}  = N(N-1)(N-2){\Qb_g}(\beta^{(k)}_N \partial_t^2(v_1^{N-3}v_2v_3v_4)).
\]
By Lemma \ref{lm_psmodel}, it has the principal symbol
\begin{align}\label{eq_psUN}
&-2 N(N-1)(N-2) (2\pi)^{-(m+1)}
{\sigmp}({Q}_g)(y, \eta, q, \zeta)  \\
& \quad \quad \quad \quad \quad \quad \quad \quad \quad \quad \quad \quad \quad \quad\quad \times
(\zeta_0)^2 \beta^{(k)}_N
\sigmp(v_1^m)(q, \zeta^{(1)})
\prod_{j=2}^{4} \sigmp(v_j)(q, \zj), \nonumber
\end{align}
at $(y, \eta)$, where $(y, \eta)$ is chosen as above. 
We note that
\begin{align*}
&\sigmp(\lge \nu, \nabla \mathcal{U}^{N, k}_N \rge|_{\partial M})(\yb, \etab)
=  \iota  \lge \nu, \eta \rge_g \sigmp (\mathcal{R})(\yb, \etab, y, \eta)  \sigmp(\mathcal{U}^{N,k}_N)(y, \eta), \quad k=1,2,
\end{align*}
where $(y_|, \eta_|)$ is the projection of $(y, \eta)$ on the boundary.
Therefore, combining equations (\ref{eq_UN}) and (\ref{eq_psUN}), at any $q \in \mathbb{W}$ we must have
$
\beta^{(1)}_N = \beta^{(2)}_N.
$

This section with Section \ref{sec_lower} proves Theorem \ref{thm2}. 
Then by Proposition \ref{pp_thm1}, we have Theorem \ref{thm1}.

\section*{Acknowledgment}
We would like to thank Katya Krupchyk for mentioning the paper \cite{Kaltenbacher2021}. 
The research of G.U. is partially supported by NSF,
a Walker Professorship at UW, a Si-Yuan Professorship at IAS, HKUST, and a Simons Fellowship.
Part of this research was performed while G.U. and Y.Z. were visiting the Institute for Pure and Applied Mathematics (IPAM),
which is supported by the National Science Foundation (Grant No. DMS-1925919). Y.Z. was also partially supported by a Simons Travel Grant.

\begin{footnotesize}
    \bibliographystyle{plain}
    \bibliography{microlocal_analysis}
\end{footnotesize}

\end{document}